\documentclass[review,10pt,authoryear]{elsarticle}

\usepackage{amssymb}
\usepackage{amsmath}
\usepackage{amsthm}

\usepackage[round]{natbib}
\usepackage[top=2cm,bottom=3cm,left=2.5cm,right=2.5cm]{geometry}

\usepackage{subfigure}

\theoremstyle{definition}
\newtheorem{definition}{Definition}[section]
\newtheorem{theorem}{Theorem}[section]
\newtheorem{corollary}{Corollary}[theorem]
\newtheorem{lemma}[theorem]{Lemma}

\DeclareMathOperator*{\argminA}{arg\,min}

\usepackage{multirow}

\usepackage{colortbl}

\usepackage{setspace}

\usepackage{natbib} 
\usepackage{etoolbox}

\AtBeginEnvironment{thebibliography}{
  \singlespacing
  \setlength{\itemsep}{0pt}
  \setlength{\parskip}{0pt}
}
\usepackage{natbib} 
\begin{document}

\begin{frontmatter}

\title{Minimizing Material Waste in Additive Manufacturing through \\  Online Reel Assignment}

\tnotetext[t1]{This research has received funding from the AIMS5.0 project, which is supported by the Chips Joint Undertaking and its members, including the top-up funding by National Funding Authorities from involved countries
under grant agreement no. 101112089.}

\author[1]{Ilayda Celenk\corref{cor1}}
\ead{i.celenk@tue.nl}
\cortext[cor1]{Corresponding author}

\author[2]{Willem van Jaarsveld}
\ead{W.L.v.Jaarsveld@tue.nl}

\author[3]{Ivo Adan}
\ead{i.adan@tue.nl}

\author[4]{Alp Akcay}
\ead{a.akcay@northeastern.edu}

\affiliation{organization={School of Industrial Engineering},
            addressline={Eindhoven University of Technology}, 
            city={Eindhoven},
            postcode={5600 MB}, 
            country={the Netherlands}}

\begin{abstract}
We study a variant of the online bin packing problem that arises in filament-based 3D printing systems operating in make-to-order settings, where only a limited number of filament reels of finite capacity can be handled at once. Components are assigned to reels upon arrival and insufficient reels are discarded to be replaced with new ones, resulting in material waste. To minimize the long-run average discarded filament through an online assignment policy, we formulate this problem as an infinite-horizon average-cost Markov Decision Process and analyze the structure of policies under stochastic, sequential demand. We first show that under a random allocation policy, the system decomposes into a collection of identical single-reel processes, allowing us to derive a closed-form expression for the average waste and enabling a tractable baseline analysis. Building on this decomposition, we construct a theoretically grounded index policy that assigns each reel a score reflecting the marginal cost of assignment and prove that it constitutes a one-step policy improvement over random allocation. We embed the index-based structure within a Deep Reinforcement Learning framework using approximate policy iteration. The resulting method achieves near-optimal performance across a range of simulated and real-world scenarios. Our results demonstrate that Reinforcement Learning policy significantly reduces material waste while maintaining real-time feasibility and interpretability.
\end{abstract}

\begin{highlights}
\item Models online reel assignment in make to order three dimensional printing
\item Introduces an index based reel assignment heuristic inspired by the Gittins index
\item Proves structural results that characterize optimal reel control decisions by single reel decomposition
\item Shows the index policy outperforms standard bin packing based assignment rules
\item Improves the index policy with deep reinforcement learning to reach near optimality
\end{highlights}


\end{frontmatter}

\section{Introduction}\label{sec:intro}

Additive manufacturing, commonly known as \emph{3D printing}, is rapidly transforming make‑to‑order(MTO) production in aerospace, medical devices, automotive, and customized consumer goods \citep{Gibson2015, Ford2020}, in which manufacturing starts after receiving a customer request, enabling delivery of highly customized products without holding large inventories. 3D printers are particularly well suited for MTO environments as they can easily switch between different product designs without the need for expensive tooling, changeovers, or long setup times. This makes producing two entirely different items back-to-back feasible and cost-effective. However, while 3D printing technology reduces physical reconfiguration costs, it shifts the challenge to real-time production planning, as customer orders typically arrive sequentially and decisions must be made without knowing future requests. This turns production planning into an online decision problem, where manufacturers must manage resources dynamically while balancing cost, quality, and efficiency. The increasing demand for customized products, coupled with intense market competition and holding costs, forces companies to incorporate techniques to ensure cost-effective use of resources.

Among the various technologies, filament based material extrusion printers are attractive because of their modest capital cost, wide material palette in terms of color, and ability to build high value items directly from digital models. Material–extrusion printers are fed from reels or spools, as shown in Figure \ref{fig:reel_of_filament}, that typically hold 2000-5000 grams of filament. 

\begin{figure}[h]
    \centering
    \includegraphics[width=0.25\textwidth]{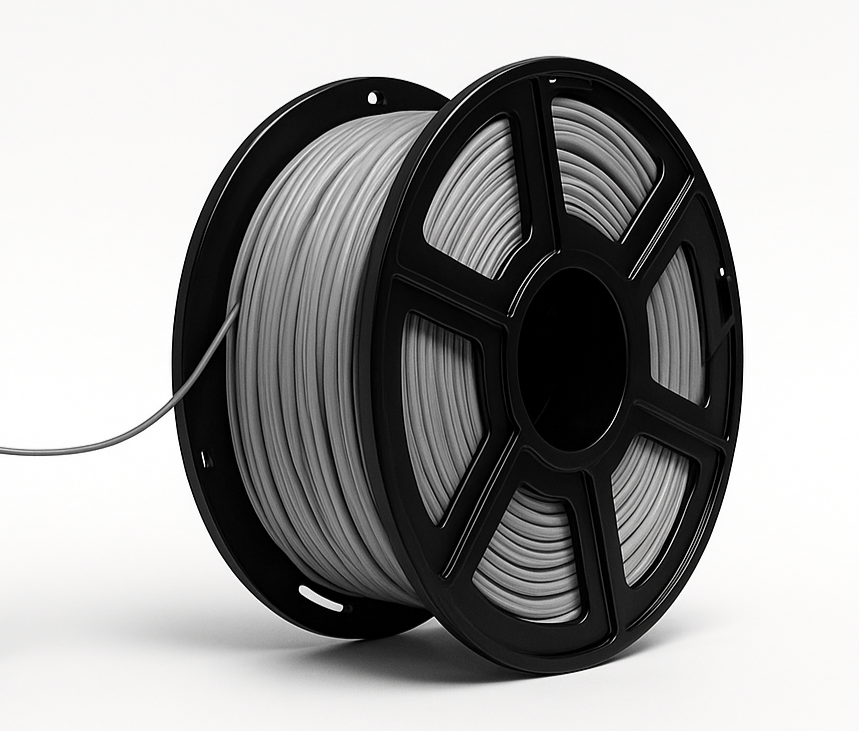} 
    \caption{Reel of Filament}
    \label{fig:reel_of_filament}
\end{figure}

Throughout the process, the filament needs to be kept at a certain temperature, allowing it to be melted slowly as it passes through the narrow tube feeding the printer's nozzle. As the reel rotates, the filament continuously feeds the machine, and the object is built up layer by layer by the movement of the nozzle. The filament on the reel gradually depletes as the object being printed grows in weight. 

Importantly, once a print job has started, swapping the reel is strongly discouraged. Changing the reel during a process can interrupt the steady flow of melted filament in the tube, leading to surface blemishes, weak inter-layer bonding, and potentially complete job failure. As a result, companies avoid changing the reel during printing to maintain component quality and avoid costly reprints. This operational constraint requires planners to assign each component production order to a reel that has sufficient remaining filament weight to complete the component without interruption. 

In practice, the number of reels that can be handled simultaneously is limited by the total number of printers therefore induces a hard upper bound on the number of active reels in the system. Once a reel’s remaining filament is insufficient to process any job, the reel is removed and discarded, creating material waste. Since such discarded leftovers accumulate over time, minimizing this waste becomes a primary objective in deciding which reels occupy the limited slots.

Due to the high variety of colors used in production and the relatively low arrival rate of component orders for each individual color, it is not possible to anticipate when the next component of a given color will arrive. As a result, production must proceed as soon as a component order arrives, with no opportunity to wait for future requests of the same color. This forces sequential decision-making for each color, where each component must be assigned to a reel immediately upon arrival. Such real time allocation decisions are necessary to maintain throughput and avoid production delays, even though they must be made without knowledge of future demands.

Efficient material usage depends on making informed reel assignments for each component while continuously tracking the remaining filament on every reel to ensure uninterrupted completion of future print jobs. Figure \ref{fig:example_assignment} illustrates a simple example involving three sequentially arriving component orders and two identical reels. Suppose the components have weights $C_1$, $C_2$, and $C_3$, and the two reels have identical initial weights $R_1$ and $R_2$. The lengths in the figure are proportional to these weights. It is assumed that no component order weighing less than $C_2$ can arrive. Components are assigned to reels sequentially, in the order of arrival, which means that $C_1$ must be assigned without knowing that $C_2$ and $C_3$ will arrive next or what their weights will be. Two possible assignments are illustrated. In both cases, the reels are left with unusable leftover filament after processing all three components, but the first option results in significantly less waste than the second. This example highlights that even in small instances, the quality of online decisions can have a substantial impact on material utilization.

\begingroup
\begin{figure}[h] 
    \centering
\subfigure[Status of the reels before order arrival.]{\label{fig:before}\includegraphics[width=5.5cm]{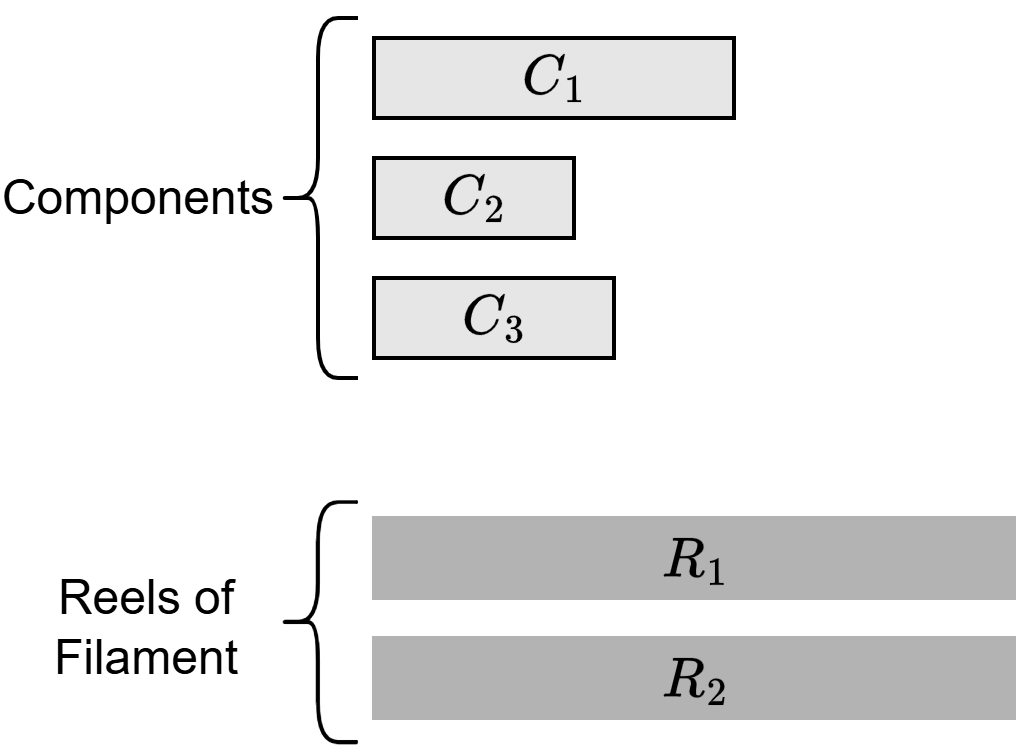}}
\hspace{0.5cm}
\subfigure[Two possible assignments after order arrivals.]{\label{fig:after}\includegraphics[width=5.5cm]{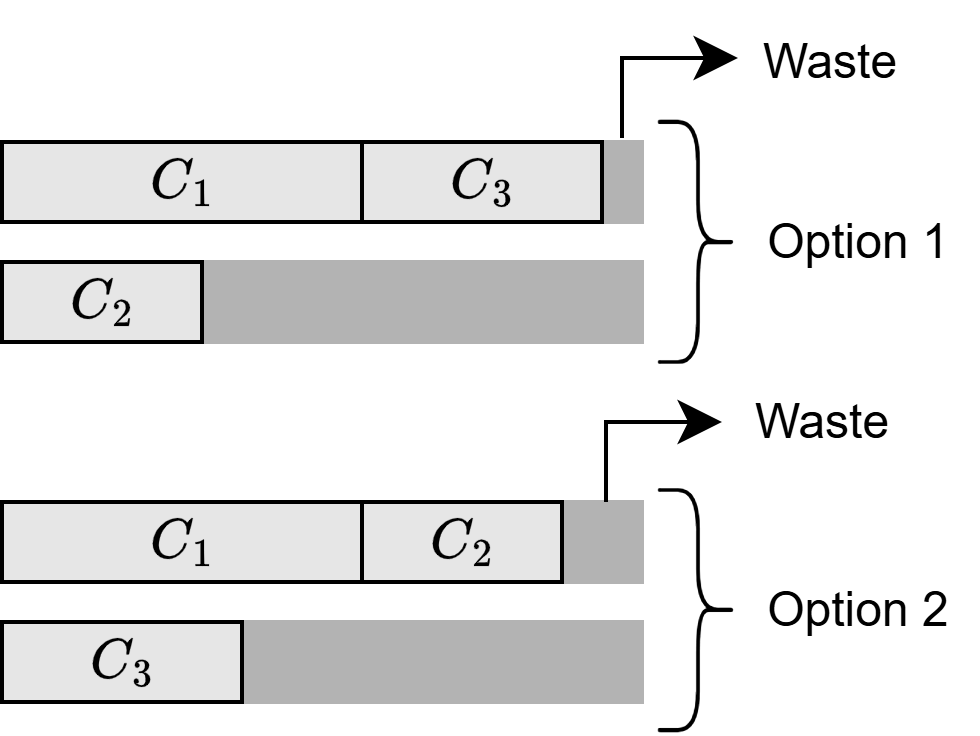}}
    \caption{Example component-to-reel assignment.}
    \label{fig:example_assignment}
\end{figure}
\endgroup

The resulting reel allocation problem in filament-based 3D printing shares similarities with one-dimensional online cutting and packing problems, such as the cutting stock problem (CSP) and the bin packing problem (BPP). In these problems, the goal is to assign sequential demand to limited-capacity resources one by one while minimizing unused capacity. Specifically, CSP focuses on cutting smaller items from larger stock materials to minimize leftover, while BPP seeks to pack items into bins to minimize the number of bins used. In the context of 3D printing, filament reels act as the stock materials or bins, and component print jobs are the cutting or packing requests. Offline problems, where all cutting requests are known in advance, have been extensively studied in the literature with comprehensive classifications and solution methods \citep{Martello1990, Dyckhoff1990, Waescher2007}.

Motivated by a real-world manufacturing environment, we consider a stylized setting in which component orders arrive sequentially with stochastic filament requirements and immediately assigned to one of the available reels. Each reel has a finite filament capacity and can only be used if it holds enough material to complete the assigned component. This leads to an online decision-making environment, where choices are irrevocable and the effect of the one-by-one assignment is not observed immediately, creating a variant requires more sophisticated methods than existing formulations of cutting-stock or bin-packing problems, which often assume full information in advance or easier conditions. These practical constraints in filament-based 3D printing give rise to a constrained resource setting, which we refer to as an N-bounded online cutting-stock problem with stochastic job sizes and capacitated bins.

In this study we address the industrial problem of minimizing filament waste in make-to-order 3D printing systems that operate with a limited number of reels and sequentially arriving component orders with stochastic weights. We model the reel assignment as an infinite-horizon average-cost Markov Decision Process and seek policies that select a reel based on the current component and reel weights. We propose an index policy, a widely used approach in sequential decision making that assigns a real-valued score to each action in a given state and selects the action with the best score, which yields interpretable and computationally efficient approximations to optimal policies.

Our first contribution is to show that, under the random allocation policy, the system dynamics decompose into independent single-reel MDPs, enabling a tractable analysis of the baseline behavior. Building on this structure, and in contrast to existing index-based methods used in bin packing problems, we derive a theoretically grounded index policy that exploits single-reel dynamics and prove that it constitutes a one-step policy improvement over random allocation. We then quantify this improvement both analytically and through empirical evaluation. Finally, we iteratively refine the index policy using Deep Reinforcement Learning. The resulting policy achieves near-optimal performance while maintaining interpretability and robust generalization across diverse component weight distributions.

This paper is organized as follows: \S\ref{sec:literature} situates our work within the cutting‑stock, online optimization, and additive‑manufacturing literature,  \S\ref{sec:Problem} summarizes the industrial setting and formalizes the MDP. Methods are developed in \S\ref{sec:policies}, numerical experiments are reported in \S\ref{sec:results}. Finally, \S\ref{sec:discussion} concludes and highlights avenues for future research.

\section{Literature Review}\label{sec:literature}

We organize the review along three streams: resource allocation in additive manufacturing,  online cutting stock and bin packing, and stochastic control frameworks with reinforcement learning.

Most research in additive manufacturing (AM) focuses on scheduling problems \citep{Anton2022, Agnol2022}, such as makespan minimization \citep{Alicastro2021}, number of tardy jobs \citep{Sormaz2023}, batch printing or nesting \citep{Pinto2024}, rather than material-aware planning. For example, \citet{Li2017} introduce one of the first mathematical formulations for production planning in AM, addressing part-to-machine assignments to minimize cost per unit volume, though material handling is not modeled. Similarly, \citet{Chergui2018} propose an integrated approach to production scheduling and nesting in additive manufacturing, combining job sequencing with 2D placement heuristics to improve build volume utilization. However, their model does not incorporate material constraints such as filament depletion or reel capacity. \citet{Aloui2021} address a machine-specific scheduling problem under technological constraints like part preheating and post-processing times, using a hybrid heuristic. Yet, similar to other scheduling approaches, their model does not track reel-level filament usage or account for filament consumption during job execution. 

Despite growing attention on AM scheduling problems, studies using hybrid heuristics, genetic algorithms, and recent reinforcement learning models focus on machine scheduling or job ordering, not filament-level constraints \citep{Sun2025, Alicastro2021}. Material planning is often abstracted away; either reels are assumed infinite, or a reel change is modeled as zero‑time, zero‑cost. In contrast, our work studies a novel but realistic constraint in filament-based AM: \emph{material planning under no mid-print reel replacement}. This setting introduces an online assignment challenge that was not captured in prior models, where material utilization must be optimized as components with stochastic weight requirements arrive sequentially and assigned to capacitated resources. To the best of our knowledge, there is currently no study that simultaneously models the dynamic utilization of reels in online additive manufacturing settings, considering stochastic arrivals, capacity limitations, and costs associated with filament discard.

The classic one-dimensional Cutting Stock Problem (CSP) and Bin Packing Problem (BPP) are foundational in combinatorial optimization aiming to minimize waste or bin usage while satisfying demand constraints. These problems have been studied extensively for both their theoretical and practical relevance in manufacturing, logistics, and scheduling. 

Early theoretical work such as that by \citet{Johnson1973} established tight bounds on competitive ratios for the online BPP, leading to practical algorithms like First-Fit (FF), Best-Fit (BF), and Next-Fit (NF) \citep{Albers2000}, and Harmonic variants \citep{Seiden2002}. FF algorithm keeps a list of bins in the order they are opened and assigns the item to the first bin it can fit. Similarly,  BF algorithm keeps a list of open bins but in an increasing order of their remaining capacity and assigns the item to the first bin it can fit. NF algorithm considers a single open bin and moves to a new one if the item doesn't fit. Harmonic-based class algorithm tries to pack items of similar sizes, and employs NF algorithm for each class. These strategies remain influential due to their simplicity and low computational overhead and a comprehensive study on exact methods developed for the online 1D-CSP and 1D-BPP can be found in the literature \citep{Delorme2016}. 

Several works extend CSP to incorporate production realities, such as due dates, machine scheduling, or heterogeneous order types. \citet{Arbib2014} study cutting stock with due dates, proposing integer programming models for balancing waste and deadline adherence. Similarly, \citet{Reinertsen2010} examine a one-dimensional CSP with due dates, highlighting the trade-off between trim loss minimization and timely order fulfillment. Although, these models bridge classic CSP with temporal scheduling, they assume that all the items assigned to the same material are completed at the same time which contradicts to the 3D printing environment we have.

In practice, MDP formulations for stochastic CSPs often arise in multi-period or dynamic contexts. For instance, \citet{PitombeiraNeto2022} formulated the stochastic CSP as an MDP to determine how many objects to cut in each pattern at each decision epoch given the current inventory levels. They developed a reinforcement learning approach using approximate policy iteration to improve their proposed heuristic. This study demonstrates the effectiveness of heuristic techniques in addressing online cutting stock problems, offering significant improvements in solution quality and computational efficiency across different problem variants. A related stream examines joint CSP and production planning. For example, \citet{Gramani2006} and \citet{Andrade2021} model integrated lot-sizing and cutting stock, showing how trim-loss and inventory dynamics interact. While these studies typically assume make-to-stock environments, they underline the value of combining capacity planning with material utilization.

Originating from multi-armed bandit problems and later extended to stochastic scheduling and resource allocation, index policies offer a balance between tractability and performance. In BPPs, differential index and ratio index policies have been proposed as interpretable and computationally cheap heuristics to prioritize assignments to minimize overfill of the bins \citep{AsgeirssonStein2009, Peeters2019}. While these approaches often perform well in practice, they are typically not derived from formal control-theoretic principles and may not guarantee improvement over naive baselines. Our work advances this line of research by deriving an index policy that corresponds to a one-step improvement over the random allocation policy, grounded in the structure of the underlying MDP.

Despite the progress, few studies address online CSP under stochastic item sizes, trim objectives, and a bounded number of bins, which is a crucial for reel utilization in 3D printing. Most algorithms optimize either worst-case waste or tardiness but do not model dynamic replacement costs or irreversible assignments as in our reel assignment problem. Despite the maturity of CSP/BPP research, no known studies to date jointly consider a bounded number of containers (reels), stochastic component sizes, and trim loss as a primary objective in an online, MTO setting. This positions our problem at the intersection of online decision-making and bounded-capacity stochastic control, filling a distinct gap between theoretical BPP/CSP studies and their limited application to additive manufacturing.

Stochastic control frameworks such as MDPs have long been used to model sequential decision-making problems under uncertainty \citep{Puterman2014}. The use of MDPs has also been extended to various manufacturing and logistics problems, including bin packing, cutting stock, and capacity management. In the context of packing and resource allocation, MDPs provide a principled way to balance immediate costs with long-term performance, especially in systems with stochastic arrivals and capacity constraints. MDP formulations typically rely on exact or approximate dynamic programming methods, which become computationally intractable as system size increases. To tackle the curse of dimensionality, approximate dynamic programming (ADP) or DRL approaches are prompted \citep{Powell2011}.  

Recent advances in reinforcement learning (RL) have provided scalable alternatives to classical dynamic programming. RL methods approximate the value function or policy through simulation-based learning, enabling practical solutions for large-scale problems. Reinforcement learning is increasingly explored for online variants of CSP. \citet{PitombeiraNeto2022} introduce an RL-based approach for stochastic cutting stock problems, showing it outperforms classic methods under uncertainty. Likewise, \citet{zhao2022} and \citet{zhao2023} formulate the online 3D bin packing problem as an MDP with constrained action space, then explore DRL solutions using an on-policy actor-critic framework. These online approaches complement deterministic metaheuristics and are especially relevant for additive manufacturing, where orders arrive sequentially and real-time decisions are critical.

Deep Reinforcement Learning (DRL) for machine scheduling is increasingly recognized as a promising research direction. \citet{Khadivi2025} provide a structured review of RL applications in additive manufacturing scheduling, identifying open challenges in integrating material constraints and online decision-making. \citet{Rinciog2022} propose a standardized framework for evaluating RL algorithms in stochastic production scheduling, emphasizing the importance of interpretability and practical feasibility.

Unlike previous work that focuses on static batch planning or stochastic online packing, we address the unique challenges of sequential material allocation in capacitated filament-based additive manufacturing. In contrast to the existing index heuristics, our approach combines structural decomposition with approximate policy improvement to bridge stochastic control and reinforcement learning, delivering both theoretical performance guarantees and scalable, real-time algorithms for reel utilization.

\section{The Online Allocation Model}\label{sec:Problem}

This section provides an overview of the practical setting, modeling assumptions, and formal structure of the reel assignment problem. \S\ref{sec:industry} describes the industrial context, highlighting the operational environment, material constraints, and motivations for online decision-making. \S\ref{sec:operatormodel} presents the operator-level decision problem, focusing on how components are sequentially assigned to reels and how filament usage and waste are tracked. \S\ref{sec:MDPmodel} formulates the MDP, defining the state and action spaces, transition dynamics, and cost structure that will be used in the development and analysis of decision policies.

\subsection{Industrial context}
\label{sec:industry}
This study is motivated by a real-world production setting at a company that manufactures custom-designed luminaries using material extrusion 3D printing technology. Customers can choose from a wide range of shapes, colors, and textures, and each component is printed only after an order is placed. To stay competitive, the company promises short delivery times while committing to sustainable production practices by using recycled materials. Customization, MTO production model, and sustainability goals create unique operational challenges that require fast and reliable decision-making during production.

To meet both functional and aesthetic requirements, the company uses polycarbonate filament as its primary printing material, which is supplied in 5-kilogram reels. Polycarbonate offers high strength, heat resistance and superior optical properties, which are essential for light-diffusing applications. It also provides excellent versatility in terms of aesthetics. It is available in a wide variety of colors and supports creating diverse surface textures through 3D printing, enabling the company to offer extensive customization even for relatively simple products. These features make polycarbonate ideal for indoor and outdoor lighting products where aesthetic quality, mechanical durability, and thermal reliability are critical. 

Despite its advantages, polycarbonate filament introduces several operational challenges, particularly with respect to material handling and process control. Unlike more commonly used thermoplastics such as PLA or PETG, polycarbonate is highly hygroscopic and degrades quickly when exposed to moisture, which can severely impact print quality. Furthermore, it requires high extrusion temperatures, increasing the risk of nozzle wear and clogging, that are not easily managed in fully automated environments. For these reasons, filament reels are loaded and changed manually by an operator to preserve print reliability and consistency. Moreover, reels cannot be swapped during a print job, as interruptions can lead to visible defects or printer failures. These constraints require that reel assignment must be done carefully before each job starts, based on available capacity and the correct color.

A further challenge is the need for online decision-making, which arises from the company’s customization options, short lead time commitments, along with its MTO production policy without inventory for finished products. Even simple designs can be produced in many different colors and textures, and many of these variants are ordered in low volumes. This results in highly fragmented material utilization across a large number of color-specific reels. Due to strict lead times, the company has only four days for printing components. This leaves no opportunity to delay production to accumulate additional component orders of the same color, which might otherwise improve efficiency. Instead, each job must be assigned to a reel immediately upon arrival, based only on the current state of the system, highlighting the online nature of the problem.

In line with the company's objectives, the central concern in this production setting is minimizing the trim loss which refers to the amount of filament left unused on a reel because no future job can be assigned to it before it is replaced. Since reels cannot be swapped during an active print job, any unwise allocation risks generating material waste. This makes it essential to keep track of the remaining filament on each reel and to make assignments that minimize such residuals. Minimizing trim loss is particularly important in this context for two key reasons. First, it directly contributes to material efficiency and waste reduction, supporting the company’s commitment to sustainable production. Second, unused filament incurs real cost, especially when color variants have low demand and leftover material cannot be reused. 

Although the company operates many printers, the number of reels that can be actively maintained at any given time is limited. This constraint arises from the manual nature of reel handling, the wide variety of color variants, and limited manpower, as well as the need to fully prepare each reel in advance. Preparation includes drying the filament, transporting the reels to 3D printers, collecting finished components. These setup and maintenance requirements introduce significant operational overhead, and when combined with limited manpower, they impose a practical limit on the number of reels that can be managed concurrently. In addition, reels are color specific, and color changes during a print job are not permitted. Therefore, the assignment problem naturally decomposes into color-specific subproblems each with specific number of reels. It is therefore both operationally realistic and analytically appropriate to model a fixed number of reels per color. 

Given the constraints imposed by material handling, customization requirements, and MTO production, the company faces a sequential decision-making problem in which each component order must be assigned to a suitable reel in real time. These assignments must balance material efficiency, operational feasibility, and adherence to aesthetic standards, all without knowing the future orders. The problem is inherently dynamic and stochastic for each color. The operator must act immediately upon the arrival of each component order, based on the current state of the system, while accounting for the uncertain weights of future jobs. Moreover, each assignment decision affects not only the selected reel but also the future usability of reels, making the long-term consequences of each choice non-trivial. A formal model that captures this sequential structure, the stochastic nature of component orders and weights, and the evolving state of the reels is therefore essential for analyzing and improving reel utilization.

\subsection{Operator model} \label{sec:operatormodel}

While reel assignment is part of a broader manufacturing workflow, we isolate this subproblem to focus on filament utilization. Thus, this study focuses exclusively on designing a policy for the assignment of components to reels with the objective of minimizing filament waste. To facilitate tractable analysis, we assume that the filament requirement for each component, expressed in terms of weight, takes values from a finite discrete set.

We develop a discrete-time model to capture the reel assignment process in a 3D printing production environment. Rather than assuming fixed-length time intervals, we abstract the process through event-driven decision epochs triggered by incoming component orders. The time between arrivals is not explicitly modeled and may vary in practice; instead, we assume that component orders arrive sequentially, independently, and identically distributed. Under this formulation, we focus on minimizing the expected filament waste per period, aligning the model’s objective directly with the efficiency of material utilization.

The operator is tasked with assigning each arriving component orders to one of the $N$ available reels. Formally, the reels are labeled $1$ through $N$, and the weight of reel $n$ in the system is denoted as $w_n \in \mathbf{W} = \{ 0, 1,\ldots, B-1\} $ where $B$ represents the weight of a new reel for $n \in \mathcal{N}=\{ 1,\ldots, N \}$. Component orders arrive sequentially, with each component’s weight independently drawn from a known discrete probability distribution. Let $X$ be the random variable over $\mathbf{X} = \{ 1,\ldots, B\}$ that represents the weight of an arriving component. The probability that the upcoming request is for a component that weighs $x$ grams is denoted by $p(x):=P(X=x)$, provided that $\sum_{x \in \mathbf{X}} p(x) = 1$. Right after the request for a component, the operator observes the remaining filament weight on the reels and component weight, then selects a reel to assign the component for printing. The operator updates the remaining filament weight on the reel in the system. However, if the chosen reel does not have enough filament for the assigned component to be printed, then it is discarded and replaced with a new reel of weight $B$ grams to be used immediately. After assigning a component of weight $x$ to a reel of weight $w$, the reel weight is updated instantaneously by

{
  \setlength{\abovedisplayskip}{4pt}
  \setlength{\belowdisplayskip}{4pt}
\begin{equation}
   e(w,x) = \begin{cases} w-x &\mbox{if } w \geq x \\
    B-x & \mbox{otherwise } \end{cases}.
    \label{eq:e(w,x)}
\end{equation}
}

In an environment with limited number of reels, filament waste is inevitable as a result of discarding a reel. Although, discarding a reel creates immediate filament waste, it may allow a better assignments in the upcoming periods, ultimately leading to less future waste. As a result of assigning a component of weight $x$ to a reel of weight $w$, the discarded reel weight is determined by

{
  \setlength{\abovedisplayskip}{4pt}
  \setlength{\belowdisplayskip}{4pt}
\begin{equation}
   c(w,x) = \begin{cases} 0 &\mbox{if } w \geq x  \\
    w & \mbox{otherwise } \end{cases}.
    \label{eq:c(w,x)}
    \end{equation}
}

The decision to discard and replace a reel trades off immediate filament waste with potential future flexibility in handling upcoming component weights. 

\subsection{MDP formulation}\label{sec:MDPmodel}

In this paper, the problem is formulated as an infinite horizon discrete time MDP. We next describe the state space, actions transition dynamics and costs of this MDP. The state space $\mathcal{S}$ captures both the weights of the reels and the weight of the arriving component. Let $x_t$ be the weight of the component arriving at the beginning of period $t$. We assume ${x_t}$ for $t = 0, 1, \ldots $ are independent and identically distributed random variables drawn from a discrete distribution $X$ over $\mathbf{X}$. The state at the beginning of period $t$ is denoted by $s_t=(w_{1t},\ldots, w_{Nt}, x_t) \in \mathcal{S} = \mathbf{W}^{N} \times \mathbf{X} $, and captures the current filament weight on all the $N$ reels just before the assignment, as well as the weight of the component to be assigned.

 Each period corresponds to a single decision epoch where the action determines the selected reel for the assignment. An action $a \in \mathcal{N}$ denotes the index of the reel selected in state $s$ to assign the component. The remaining weight on the selected reel needs to be updated after assignment. This can be broken down into two cases based on the component and selected reel weights. In case the weight of the component is smaller than the selected reel's weight, the weight of the reel decreases by the component's. However, if the amount of filament on the selected reel is not sufficient for the component to be printed, it needs to be replaced with a new one of weight $B$ before printing. Based on these two cases, the weight of the selected reel is updated to $e(w_a,x)$ by Equation \eqref{eq:e(w,x)}. After the assignment, the next component order of weight $x_{t+1} \sim X$ arrives. Thus, the state of the system at the beginning of the next period $t+1$ is $s_{t+1}=(w_{1t},\ldots, e(w_{a_t},x_t),\ldots,w_{Nt}, x_{t+1})$.

The cost of waste is incurred only when the assignment requires a reel replacement. In this case, the insufficient reel is discarded. The cost of action $a$ in state $s_t$ is denoted by $\mathcal{C}(s_t,a_t)=c(w_{a_t}, x_t)$ by Equation \eqref{eq:c(w,x)}. We aim to  minimize the expected average cost per period $g=\lim_{T \rightarrow \infty} \frac{1}{T} \sum_{t=1}^T \mathbb{E} [\mathcal{C}(s_t,a_t)]$.

The system evolves through discrete decision epochs in an infinite horizon setting. Starting with state $s^t = (w_{1t},\ldots, w_{Nt}, x_t)$ at the beginning of decision epoch $t$, Figure \ref{fig:model} demonstrates how the state changes based on the action $a_t$. At each decision epoch we update the remaining weight on the reels based on the action. The action specifies the index of the selected reel and only that reel is adjusted, while the others remain unchanged. 

\begin{figure}[h]
    \centering
    \includegraphics[width=0.6\textwidth]{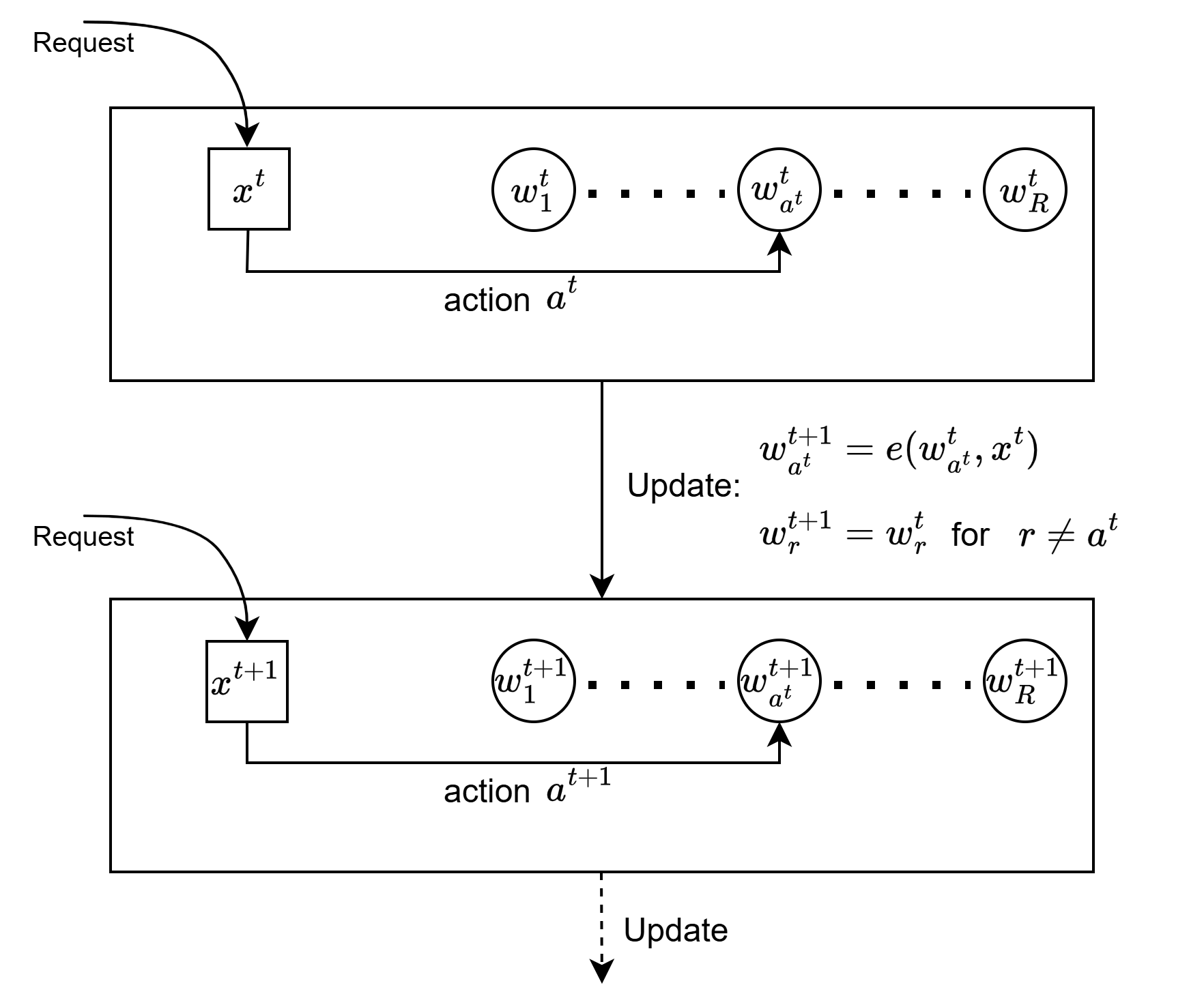} 
    \caption{MDP Model}
    \label{fig:model}
\end{figure}

\section{Methods}\label{sec:policies}

In this section, we present a range of solution approaches for the reel assignment problem. We first introduce several heuristic policies in \S\ref{sec:AnalyticalPolicies}, including First Fit and Best Fit, which rely on straightforward allocation rules based on current reel capacities.

\S\ref{sec:RandomPolicy} discusses the Random policy, where reels are selected uniformly at random. We analyze this policy in more detail to gain structural insights into single-reel behavior. This analysis allows us to decompose the system into analytically tractable single-reel processes. And in \S\ref{sec:IndexPolicy}, we leverage these insights to construct an Index policy, which uses a value function derived from the single-reel model to prioritize reels based on their long-term waste. Finally, Section~\ref{sec:DRLPolicy} introduces a DRL approach that approximates the optimal policy through simulated experience and function approximation, enabling adaptive and scalable decision-making in large state spaces.

\subsection{Analytical Policies}
\label{sec:AnalyticalPolicies}
The reel assignment problem falls within the class of combinatorial optimization problems, where the number of possible action sequences grows exponentially with the number of components and available reels. In such settings, computing the optimal policy becomes intractable, especially under online constraints where decisions must be made immediately and without knowledge of future arrivals. As a result, fast and interpretable approximate methods are essential for practical deployment.

In many online resource allocation and cutting stock problems, simple heuristic policies are widely used due to their ease of implementation and interpretability. These heuristics operate based on immediate information and follow predefined assignment logic that can be applied without extensive computation or training. In our setting, we adopt two classical heuristics: \textbf{First Fit} and \textbf{Best Fit}. However, unlike classical bin-packing formulations that allow opening new bins, our model considers a fixed number of $N$ reels. Due to this constraint, we consider the adapted versions of these heuristics.

The First Fit heuristic assigns the incoming component to the first reel, in a fixed ordering, that has enough capacity to accommodate it. If there is no feasible reel, First Fit replaces the first reel that would incur the least immediate cost upon replacement. This method is computationally fast as it does not iterate over the whole reel set and easy to implement but may lead to suboptimal long-term behavior as it does not compare reels unless reel change is necessary.

The Best Fit heuristic selects the reel with the smallest non-negative residual capacity after assigning the current component. If multiple reels are sufficient to print the component, the reel with the least remaining weight is chosen to help preserving reels with larger capacity for future components. If no reel has sufficient remaining filament to print the component, a reel must be replaced. In this case, Best Fit chooses the reel whose replacement would incur the least immediate waste. This logic prioritizes efficient usage of partially filled reels and avoids wasteful replacements when possible. 

These heuristics are frequently used in practice and often serve as competitive baselines. However, they are inherently myopic as they optimize for the current decision without modeling future implications or incorporating stochastic knowledge of incoming components. In contrast to dynamic policies derived from value functions, these rules lack foresight and are not guaranteed to perform well in the long run. 

To better understand the performance limitations of such heuristic policies and to design improved strategies, we next analyze the Random Policy. Although uniformly random allocation is suboptimal, it offers two key advantages: (i) it provides a tractable benchmark with analytically derivable cost and value functions, and (ii) it exposes structural properties of the system that can be exploited to construct more sophisticated policies that are grounded in long-run optimization.

\subsection{Random Policy}
\label{sec:RandomPolicy}

The random assignment policy, denoted by $\pi(R)$, corresponds to the case where each incoming component is assigned to one of the $N$ available reels with equal probability. The purpose of this analysis is to understand the cost structure and dynamics under random assignment in order to build the foundation for a performance-improving index policy that will be developed in subsequent sections. Although suboptimal, random assignment policy provides a valuable theoretical baseline and allows us to decompose the multi-reel system into analytically tractable subproblems. 

To support our analysis, we introduce two variations of the system that differ in the timing of information availability. The \textit{augmented model} represents the original MDP as formally defined in Section~\ref{sec:MDPmodel}, where the weight of the incoming component is known before selecting a reel. In contrast, the \textit{naive model} is a simplified version in which a reel is selected prior to observing the component weight. Although the naive model does not reflect the real-time decision-making structure of the actual setting, it serves as a valuable analytical device. In particular, it provides a simplified and intuitive view of the system’s evolution, which helps reveal structural properties of the problem. Moreover, it enables us to derive explicit expressions and formulate the augmented model in a principled way. Thus, the naive model serves as a conceptual and analytical device that supports both our theoretical analysis and policy design.

Although the decision context differs between the two models, both evolve as Markov Reward Processes (MRP) under the random policy, driven by identical stochastic elements; uniform reel selection and distribution of component weights. Moreover, the naive model under random policy provides a component-free model that can be practically decomposed into single reel problems. More significantly, it allows us to construct an explicit formulation for the augmented problem with multiple reels. 

In the augmented $N$-reel MRP, the system state includes the current component weight, reflecting the full information setting defined in our original MDP as $(w_1, \ldots, w_N, x)$. Since $x$ is observed prior to reel assignment, the expected cost and next state can be computed conditioned on this value. While the state transitions depend on both the randomness of the policy and the component distribution, the immediate cost depends only on the selected reel and the known value of $x$, and is thus unaffected by the overall distribution $p(x)$. 

\begin{definition}
\textit{Define $\hat{h}_N^{\pi(R)}(w_1, \ldots, w_N, x)$ as the bias function and $\hat{g}_N^{\pi(R)}$ as the long-run average cost associated with the augmented $N$-reel model under $\pi(R)$. These functions satisfy the Bellman equation:}
{
  \setlength{\abovedisplayskip}{4pt}
  \setlength{\belowdisplayskip}{4pt}
\begin{equation}
    \hat{h}_N^{\pi(R)}(w_1,\ldots,w_N,x) =  - \hat{g}_N^{\pi(R)} + \frac{1}{N}  \sum_{n=1}^N \Big[ c(w_n,x) + \sum_{x'} p(x')  \hat{h}_{N}(w_1,\ldots, e(w_n,x),\ldots, w_N,x')  \Big]  
    \label{eq:bellmann_hNhat}
\end{equation}
}
\end{definition}

In the naive $N$-reel model, the component weight $x$ is not observed at the time of reel selection, and the state consists only of the reel weights $(w_1, \ldots, w_N)$. This represents a setting with limited information, where the assignment decision must be made without knowing the specific requirement of the incoming component. Consequently, both the expected immediate cost and the state transition are influenced by the uniform reel selection and distribution of component weights.

\begin{definition}
\textit{Define $h_N^{\pi(R)}(w_1, \ldots, w_N)$ as the bias function and $g_N^{\pi(R)}$ as the long-run average cost associated with the naive $N$-reel model under $\pi^{R}$. These functions satisfy the Bellman equation:}
{
  \setlength{\abovedisplayskip}{4pt}
  \setlength{\belowdisplayskip}{4pt}
\begin{equation}
    h_N^{\pi(R)}(w_1,\ldots,w_N) =  - g_N^{\pi(R)} + \frac{1}{N} \sum_{n=1}^N \sum_{x'} p(x') \Big[ c(w_n,x') + h_N^{\pi(R)}(w_1,\ldots, e(w_n,x'),\ldots, w_N) \Big] .
    \label{eq:bellmann_hN}
\end{equation}
}
\end{definition}

Although the two MRPs differ in their state representations, they are exposed to the same stochastic dynamics. In both of the processes, the costs use the same independent and identically distributed sequence of component weights $\{ x_t \sim p(x) \}$. The difference between the Bellman equations lies in the timing of observing the component weight $x$. Therefore, in the long run, both processes observe the same action-cost realizations with the same reel weights. As such, they share the same long-run average cost (gain). We formalize this result by the following lemma.

\begin{lemma}
Consider the random policy $\pi(R)$ under which each reel is selected with equal probability of $\frac{1}{N}$. Let $\hat{g}_N^{\pi(R)}$ and $g_N^{\pi(R)}$ denote the long-run average costs of the augmented and naive $N$-reel MRPs under $\pi(R)$, respectively, with Bellman equations given by \eqref{eq:bellmann_hNhat} and \eqref{eq:bellmann_hN}. Then,
{
  \setlength{\abovedisplayskip}{4pt}
  \setlength{\belowdisplayskip}{4pt}
    \begin{equation}
        \hat{g}_N^{\pi(R)} = g_N^{\pi(R)}
    \qquad \forall N \in \mathbb{Z^+}.
    \label{eq:gN_gNhat}
    \end{equation}
    }
    \label{lemma:gN_gNhat}
\end{lemma}
\vspace{-2\baselineskip}
\begin{proof}
We define the long-run average costs under the random policy $\pi(R)$ in both MRPs using their one-step costs. In the augmented $N$-reel MRP under $\pi(R)$, the current component weight $x_t$ is observed at time $t$, and the one-step cost is given by:
{
  \setlength{\abovedisplayskip}{4pt}
  \setlength{\belowdisplayskip}{4pt}
\begin{equation*}
\frac{1}{N} \sum_{n=1}^N c(w_{n,t}, x_t)
\end{equation*}
}
where the randomness comes from the distribution $x_t \sim p(x)$.
The long-run average cost is therefore:
{
  \setlength{\abovedisplayskip}{4pt}
  \setlength{\belowdisplayskip}{4pt}
\begin{equation}
    \hat{g}_N^{\pi(R)} = \lim_{T \to \infty} \frac{1}{T} \sum_{t=1}^T \mathbb{E}_{x_t} \left[ \frac{1}{N} \sum_{n=1}^N c(w_{n,t}, x_t) \right]
    \label{eqn:g_N_hat-defn}
\end{equation}
}
In the naive model, the reel is selected before observing the component, so the one-step cost is computed by taking the expectation over $x \sim p(x)$ in advance
{
  \setlength{\abovedisplayskip}{4pt}
  \setlength{\belowdisplayskip}{4pt}
\[
\frac{1}{N} \sum_{n=1}^N \sum_{x \in \mathbf{X}} p(x) c(w_{n,t}, x) = \mathbb{E}_{x_t} \left[ \frac{1}{N} \sum_{n=1}^N c(w_{n,t}, x_t) \right].
\]
}
Thus, its long-run average cost is
{
  \setlength{\abovedisplayskip}{4pt}
  \setlength{\belowdisplayskip}{4pt}
\begin{equation}
    g_N^{\pi(R)} = \lim_{T \to \infty} \frac{1}{T} \sum_{t=1}^T \mathbb{E}_{x_t} \left[ \frac{1}{N} \sum_{n=1}^N c(w_{n,t}, x_t) \right].
    \label{eqn:g_N-defn}
\end{equation}
}
Since both formulations evaluate the same per-step cost structure and evolve under the same randomness, their long-run average costs provided in equations \ref{eqn:g_N_hat-defn} and \ref{eqn:g_N-defn} are identical, thus
{
  \setlength{\abovedisplayskip}{4pt}
  \setlength{\belowdisplayskip}{-5pt}
\[
\hat{g}_N^{\pi(R)} = g_N^{\pi(R)}.
\]
}
\end{proof}
\vspace{-0.5\baselineskip}
This result shows that knowledge on the current component to be allocated does not change the long-run average cost under random allocation. It confirms that the random policy $\pi(R)$ leads to the same stationary average cost whether the system operates in an augmented or naive state space. This property allows us to use the simpler naive model to derive bias function structures and informs the design of index policies that improve upon the random allocation baseline.

Although the $N$-reel system involves joint dynamics over multiple interacting reels, under the random policy $\pi(R)$ all reels evolve independently, which allows for a powerful simplification. Under random allocation, each reel is selected independently with uniform probability, and all reels are individually exposed to the same sequence of component weights in expectation. While changes in individual reel weights are not strictly independent each reel evolves under statistically identical conditions. This symmetry motivates the analysis of a single-reel MRP, which tracks the behavior of a single reel in isolation. The key insight is that random allocation over $N$ reels effectively induces $N$ identical and decoupled single-reel MRPs. As a result, the behavior and cost of the full system can be understood through the analysis of just one reel. The following definition formalizes the single-reel model that captures this behavior.

\begin{definition}
\textit{Let $g_1$ and $h_1(w)$ be any solution to the naive single-reel MRP, defined by the Bellman equation:}
{
  \setlength{\abovedisplayskip}{4pt}
  \setlength{\belowdisplayskip}{4pt}
\begin{equation}
    h_{1}(w) = - g_{1} + \sum_{x'} p(x') \Big( c(w,x') +  h_{1}(e(w,x')) \Big) \quad \forall w.
    \label{eq:bellmann_h1}
\end{equation}
}
\end{definition}

In the naive single-reel MRP, the state is given by the current remaining weight $w \in \mathbf{W}$ on the reel. The function $h_1(w)$ represents the bias, while $g_1$ is the corresponding long-run average cost. The bias function $h_1(w)$, derived under random allocation, reflects the relative value of having $w$ grams of filament on a reel and can thus be used to construct index-based assignment policies. It quantifies how favorable or costly it is to be in state $w$ relative to the long-run average behavior. From another perspective, $h_1(w)$ is the expected waste of a reel with weight $w$ assuming that the reel is consumed for components with weight distribution $p(x)$. It is important to note that in this special case the action is deterministic. The state transitions only depend on the current reel weight $w$ and the randomness of the component weight $x' \sim p(x)$. The randomness arises solely from the distribution of component weights. 

Intuitively, this simplification preserves the long-run average cost under random policy. The cost incurred in the $N$-reel system under random assignment is equal to the cost in the single-reel model. The following lemma formalizes this result and shows that the long-run average cost $g_N$ is invariant not only to the number of reels $N$, but also to how many reels are tracked simultaneously.

\begin{lemma}
Suppose $\pi(R)$ is the random policy, with equal action probabilities for all reels. Let $g_N^{\pi(R)}$ denote the long-run average cost associated with $\pi^R$ for the naive $N$-reel system, and $g_1$ denote the long-run average cost for the corresponding single-reel MRP. Then, for any $N \in \mathbb{Z}^+$
{
  \setlength{\abovedisplayskip}{4pt}
  \setlength{\belowdisplayskip}{4pt}
\begin{equation*}
     g_N^{\pi(R)} = g_1.
\end{equation*}
}
\label{lemma:gN_g1}
\end{lemma}
\vspace{-1.9\baselineskip}
\begin{proof}
By definition, the long-run average cost for the single-reel MRP is 
{
  \setlength{\abovedisplayskip}{4pt}
  \setlength{\belowdisplayskip}{4pt}
\begin{equation} 
g_1 = \lim_{T \rightarrow \infty} \frac{1}{T} \sum_{t=1}^T \mathbb{E}_{x_t} [c(w_t, x_t)]. 
\label{eq:g_1} 
\end{equation}
}
For the naive $N$-reel system under the random policy, the long-run average cost is
{
  \setlength{\abovedisplayskip}{4pt}
  \setlength{\belowdisplayskip}{4pt}
\[
    g_N^{\pi(R)} = \lim_{T \rightarrow \infty} \frac{1}{T} \sum_{t=1}^T \mathbb{E}_{x_t} \left[ \frac{1}{N} \sum_{n=1}^N c(w_{n,t}, x_t) \right].
\]
}
Since each reel is selected uniformly at random and independently of state, and all reels evolve identically in distribution under $\pi^R$, interchanging summation and expectation, we get:
{
  \setlength{\abovedisplayskip}{4pt}
  \setlength{\belowdisplayskip}{4pt}
\begin{align*}
    g_N^{\pi(R)} &= \frac{1}{N} \sum_{n=1}^N \lim_{T \rightarrow \infty} \frac{1}{T} \sum_{t=1}^T \mathbb{E} [ c(w_{n,t}, x_t) ] \\
    &= \frac{1}{N} \sum_{n=1}^N g_1 = g_1.
\end{align*}
}
This implies $g_N^{\pi(R)} = g_1$ for all $N$.
\end{proof}

Moreover, we have $g_N^{\pi(R)} = g_{N'}^{\pi(R)}$ for any $N, N' \in \mathbb{Z}^+$. This results shows that the long-run average cost of the naive system under the random policy is independent of the number of reels. This lemma also justifies using the single-reel MRP to reason about average cost behavior in larger systems, thus provides a theoretical foundation for index-based policies. 

 We now analyze how the structure of the bias function in the $N$-reel system relates to that of the single-reel system. Specifically, under the random allocation policy, each reel experiences the same stochastic environment and contributes equally to the system's overall cost dynamics. This symmetry suggests that the bias function for the $N$-reel MRP can potentially be represented as an aggregation of single-reel bias functions provided that we consider random allocation policy. The following theorem formalizes this intuition by showing that the bias function $h_N(w_1, \ldots, w_N)$ can be expressed as the sum of the individual reel bias functions $h_1(w_n)$.

\begin{theorem}
Let $h_N^{\pi(R)}$ be the bias function and $g^{\pi(R)}$ be the gain associated with the random policy $\pi(R)$ for the naive $N$-reel MDP. Then for any state $(w_1,\ldots,w_N)$,
{
  \setlength{\abovedisplayskip}{4pt}
  \setlength{\belowdisplayskip}{4pt}
\begin{equation*}
     h_N^{\pi(R)}(w_1,\ldots,w_N) = \sum_{n=1}^N h_1(w_n) 
     \label{eq:h_N+=h_1}
\end{equation*}
}
is a solution to the Bellman equation given in \eqref{eq:bellmann_hN}.
\label{thm:h_N+=h_1}

\end{theorem}

\begin{proof}

    Let $(w_1,\ldots,w_N)$ be an arbitrary state of the naive $N$-reel MRP and $(w_1, \ldots, e(w_n, x'), \ldots, w_N)$ be the state at the next step for an arbitrarily chosen reel $n\in{1, ...,N}$. We want to verify that 
    {
  \setlength{\abovedisplayskip}{4pt}
  \setlength{\belowdisplayskip}{4pt}
    \begin{align}
        h_N^{\pi(R)}(w_1,\ldots,w_N) = \sum_{n=1}^N h_1(w_n) 
        \label{eq:thm_eq1}
    \end{align}
    }
    and
    {
  \setlength{\abovedisplayskip}{4pt}
  \setlength{\belowdisplayskip}{4pt}
    \begin{align}
        h_N^{\pi(R)}(w_1, \ldots, e(w_n, x'), \ldots, w_N) = h_1(e(w_n, x')) + \sum_{n=1}^N h_1(w_n) - h_1(w_n) \quad \forall n
        \label{eq:thm_eq2}
    \end{align}
    }
    satisfy the Bellman equation of $h_N^{\pi(R)}$ that is given by
    {
  \setlength{\abovedisplayskip}{4pt}
  \setlength{\belowdisplayskip}{4pt}
    \begin{equation*}
    h_{N}^{\pi(R)}(w_1,\ldots,w_N) + g_{N}^{\pi(R)} - \frac{1}{N} \sum_{n=1}^N \sum_{x'} p(x') \Big( c(w_n,x') + h_{N}^{\pi(R)}(w_1,\ldots, e(w_n,x'),\ldots, w_N) \Big) = 0.
\end{equation*}
}
    Showing that the following expression is zero for any $w_n$ will complete the proof:
    {
  \setlength{\abovedisplayskip}{4pt}
  \setlength{\belowdisplayskip}{4pt}
    \begin{align}
        \sum_{n=1}^N h_1(w_n)  + g_N^{\pi(R)} - \frac{1}{N}\sum_{n=1}^N \sum_{x'} p(x') \Big[ c(w_n,x') + h_1(e(w_n, x')) + \sum_{n=1}^N h_1(w_n) - h_1(w_n) \Big].
        \label{expression:h_N}
    \end{align}
}

    First, to get an alternative representation of $\sum_{n=1}^N h_1(w_n)$ which occurs on the right hand side of the equation \eqref{eq:thm_eq1}, we use Bellman equation for $h_1(w_n)$ given by \ref{eq:bellmann_h1}:
    {
  \setlength{\abovedisplayskip}{4pt}
  \setlength{\belowdisplayskip}{4pt}
    \begin{align*} 
    \sum_{n=1}^N h_1(w_n) &= \sum_{n=1}^N \left( -g_1 + \sum_{x'} p(x') \left( c(w_n, x') + h_1(e(w_n, x')) \right) \right) \\ &= -N g_1 + \sum_{n=1}^N \sum_{x'} p(x') \left( c(w_n, x') + h_1(e(w_n, x')) \right). \end{align*}
}
    Dividing both sides by $N$ to get:
    {
  \setlength{\abovedisplayskip}{4pt}
  \setlength{\belowdisplayskip}{4pt}
    \[
\frac{1}{N} \sum_{n=1}^N h_1(w_n) = - g_1  + \frac{1}{N} \sum_{n=1}^N \sum_{x'} p(x') \Big[ c(w_n,x') +  h_{1}(e(w_n,x')) \Big]
\]
}
    Adding $\sum_{n=1}^N h_1(w_n) - \frac{1}{N} \sum_{n=1}^N h_1(w_n)$ to both sides of the equation, we get:
    {
  \setlength{\abovedisplayskip}{4pt}
  \setlength{\belowdisplayskip}{4pt}
    \begin{align}
         \sum_{n=1}^N h_1(w_n) & = - g_1  + \frac{1}{N}\sum_{n=1}^N \sum_{x'} p(x') \Big[ c(w_n,x') +  h_{1}(e(w_n,x')) \Big] + \sum_{n=1}^N h_1(w_n) - \frac{1}{N}\sum_{n=1}^N h_1(w_n) .
         \label{eq:thm_eq_1alt}
    \end{align}
}
    Equation \ref{eq:thm_eq_1alt} provides an alternative representation to $\sum_{n=1}^N h_1(w_n)$, which we can substitute for the first term in expression \eqref{expression:h_N}. Rewriting this expression then canceling out the long-run average cost terms by Lemma \ref{lemma:gN_g1} and using $\sum_{x'} p(x') = 1$ we get
{
  \setlength{\abovedisplayskip}{4pt}
  \setlength{\belowdisplayskip}{4pt}
\begin{align*}
        & - g_1  + \frac{1}{N}\sum_{n=1}^N \sum_{x'} p(x') \Big[ c(w_n,x') +  h_{1}(e(w_n,x')) \Big] + \sum_{n=1}^N h_1(w_n) - \frac{1}{N}\sum_{n=1}^N h_1(w_n)  \\ & + g_N^{\pi(R)} - \frac{1}{N}\sum_{n=1}^N \sum_{x'} p(x') \Big[ c(w_n,x') + h_1(e(w_n, x')) + \sum_{n=1}^N h_1(w_n) - h_1(w_n) \Big] =  0.
    \end{align*}
}
Since this expression is $0$, $\sum_{n=1}^N h_1(w_n)$ satisfies the Bellman equation for $h_N^{\pi(R)}$, with $g_N^{\pi(R)} = g_1$.
\end{proof}

We have shown that the bias function $h_N^{\pi(R)}$ of the naive $N$-reel system can be decomposed into the sum of single-reel bias functions $h_1$ under the random policy. This result demonstrates that the complexity of the $N$-reel system under random allocation can be significantly reduced by analyzing the behavior of a single reel. The additive structure of $h_N^{\pi(R)}$ highlights the distinct and independent contribution of each reel to the overall value function $h_N^{\pi(R)}$. This is evident even when taking into account the coupling that arises from the randomness inherent in the selection process. This structural property is particularly valuable for designing index policies, as it enables the evaluation of the marginal benefit of assigning a component to a particular reel based only on its individual state. 

Now, we will focus on the augmented $N$-reel system which directly corresponds to the MDP described in Section \ref{sec:MDPmodel}. In this setting, the current component weight $x$ is known at the time of decision-making. We claim that the bias function $\hat{h}_N^{\pi(R)}$ of the augmented system can also be explicitly represented in terms of $h_N^{\pi(R)}$ and immediate costs.

\begin{theorem}
For any $\theta \in \mathbb{R}$, 
{
  \setlength{\abovedisplayskip}{4pt}
  \setlength{\belowdisplayskip}{4pt}
\begin{equation}
     \hat{h}_N^{\pi(R)}(w_1,\ldots,w_N, x) = \frac{1}{N} \sum_{n=1}^N \Big( c(w_n,x) + h_N^{\pi(R)}(w_1,\ldots, e(w_n,x),\ldots, w_N)  \Big) + \theta
     \label{eq:h_Nhat+=h_1}
\end{equation}
}
is a solution to the Bellman equation given in \eqref{eq:bellmann_hNhat}.

\label{thm:h_Nhat+=}
\end{theorem}

\begin{proof}

Let $\theta \in \mathbb{R}$ be an arbitrary constant. We substitute the proposed form of $\hat{h}_N^{\pi(R)}(w_1, \ldots, w_N, x)$ into the left hand side of the following Bellman equation of $\hat{h}_N^{\pi(R)}$:
{
  \setlength{\abovedisplayskip}{4pt}
  \setlength{\belowdisplayskip}{4pt}
\begin{equation*}
    \hat{h}_{N}^{\pi(R)}(w_1,\ldots,w_N,x) + \hat{g}_N^{\pi(R)} - \frac{1}{N}  \sum_{n=1}^N \Big[ c(w_n,x) + \sum_{x'} p(x')  \hat{h}_N^{\pi(R)}(w_1,\ldots, e(w_n,x),\ldots, w_N,x')  \Big] = 0
\end{equation*}
}
and get the expression
{
  \setlength{\abovedisplayskip}{4pt}
  \setlength{\belowdisplayskip}{4pt}
\begin{multline}
 \frac{1}{N} \sum_{n=1}^N \Big( c(w_n,x) + h_N^{\pi(R)}(w_1,\ldots, e(w_n,x), \ldots, w_N)  \Big)  + \theta  + \hat{g}_{N}^{\pi(R)} \\ 
- \frac{1}{N} \sum_{n=1}^N \left[ c(w_n, x) + \sum_{x'} p(x') \left( \frac{1}{N} \sum_{m=1}^N \left( c(w_m', x') + h_N^{\pi(R)}(w_1', \ldots, e(w_m', x'), \ldots, w_N') \right) + \theta \right) \right]
\label{expression:h_hat_N}
\end{multline}
}

where $w_m' = w_n$ for $m \neq n$, and $w_m' = e(w_n, x)$ for $m = n$. Showing that this expression is zero will complete the proof.

Simplifying the expression we get:
{
  \setlength{\abovedisplayskip}{4pt}
  \setlength{\belowdisplayskip}{4pt}
\begin{multline}
           \frac{1}{N} \sum_{n=1}^N  h_N^{\pi(R)}(w_1,\ldots, e(w_n,x),\ldots, w_N)  + \theta + \hat{g}_{N}^{\pi(R)} \\ 
           - \frac{1}{N^2} \sum_{n=1}^N \sum_{m=1}^N \sum_{x'} p(x')  \left( c(w_m', x') + h_N^{\pi(R)}(w_1', \ldots, e(w_m', x'), \ldots, w_N') \right) - \theta.
          \label{expression:h_hat_N_2}
    \end{multline}
}

Since $h_N$ satisfies its Bellman equation we have
{
  \setlength{\abovedisplayskip}{4pt}
  \setlength{\belowdisplayskip}{4pt}
\begin{align*}
h_N^{\pi(R)}(w_1, \ldots, e(w_n, x), \ldots, w_N) = -g_N^{\pi(R)} + \frac{1}{N} \sum_{m=1}^N \sum_{x'} p(x') \left( c(w_m', x') + h_N^{\pi(R)}(w_1', \ldots, e(w_m', x'), \ldots, w_N') \right).
\end{align*}
}
Using this we can replace the double inner sum in expression \eqref{expression:h_hat_N_2} by $h_N^{\pi(R)}(w_1, \ldots, e(w_n, x), \ldots, w_N) + g_N^{\pi(R)}$ to get
{
  \setlength{\abovedisplayskip}{4pt}
  \setlength{\belowdisplayskip}{4pt}
    \begin{align*}
          \frac{1}{N} \sum_{n=1}^N  h_N^{\pi(R)}(w_1,\ldots, e(w_n,x),\ldots, w_N)  + \theta + \hat{g}_{N}^{\pi(R)} 
           - \frac{1}{N} \sum_{n=1}^N \left( h_N^{\pi(R)}(w_1, \ldots, e(w_n, x), \ldots, w_N) + g_N^{\pi(R)} \right) - \theta = 0.
    \end{align*}
}
As $\hat{g}_N^{\pi(R)} = g_N^{\pi(R)} $ by Lemma \ref{lemma:gN_gNhat}, this expression is zero. This confirms that $\hat{h}_N^{\pi(R)}$ as defined satisfies the Bellman equation for any $\theta$, completing the proof.
\end{proof}

By Theorem \ref{thm:h_Nhat+=}, we have established an explicit relationship between the bias functions $\hat{h}_N$ and $h_N$, corresponding to the augmented and $N$-reel MRPs, respectively. Building on this result, we can now directly express $\hat{h}_N$ in terms of the single-reel bias function $h_1$. The following corollary formalizes this connection.

\begin{corollary}
For any $\theta \in \mathbb{R}$, 
{
  \setlength{\abovedisplayskip}{4pt}
  \setlength{\belowdisplayskip}{4pt}
\begin{equation}
     \hat{h}_N^{\pi(R)}(w_1,\ldots,w_N, x) = \frac{1}{N} \sum_{n=1}^N \Big( c(w_n,x) + h_1(e(w_n,x)) + (N-1) h_1(w_n)  \Big) + \theta
     \label{eq:h_Nhat+=h_1_corollary}
\end{equation}
}
is a solution to the Bellman equation given in \eqref{eq:bellmann_hNhat}. 

\label{cor:h_Nhat+=h_1_corollary}
\end{corollary}

\begin{proof}
    The proof directly follows from Theorem \ref{thm:h_N+=h_1} and Theorem \ref{thm:h_Nhat+=}.
\end{proof}

This corollary demonstrates that the augmented bias function under the random policy $\hat{h}_N^{\pi(R)}$ can be entirely described in terms of individual reel bias functions $h_1$, along with immediate costs. This decomposition is critical for developing practical index policies, as it supports localized decision-making based on per-reel states while ensuring alignment with system-wide performance.

\subsection{Index Policy}
\label{sec:IndexPolicy}

A particularly effective class of approximate decision rules for complex resource allocation problems are index-based heuristics, which assign each item to a resource by evaluating a real-valued index function that captures the priority or desirability of each option given the current state. The core idea is to measure how each decision impacts the state's future value or loss and to choose the most favorable option accordingly. While the approach is heuristic, it is conceptually related to Gittins-style index policies in that it decomposes the global problem into per-unit subproblems to assign interpretable priorities. Our proposed index policy follows this principle by constructing per-reel indices derived from a simplified, single-reel model under random allocation. However, classical Gittins indices arise in optimal stopping problems with discounted rewards, whereas our setting focuses on minimizing expected waste per assignment without discounting or continuation decisions. As a result, the index we derive serves as a computationally efficient and interpretable approximation, tailored to the structure of the reel assignment problem.

A notable example is the differential index policy, introduced in the context of bin covering and stochastic packing problems \citep{Asgeirsson2014, Peeters2022}. These policies typically define a scalar loss function over the current load or capacity of a bin, and select the option that maximizes the expected reduction in this loss. In contrast to existing applications of index policies, in filament based 3D printing, reels are never initialized with full weight since they are used immediately after replacement. This operational distinction requires a new index policy formulation that leverages structural insights specific to the problem. Furthermore, unlike previous studies, we provide theoretical performance guarantees for the resulting index policy and explicitly quantify its improvement over the random policy, thereby strengthening its practical relevance for real-world deployment.

Motivated by this reasoning, we now define a well-founded value-function-based index policy $\pi(I)$ for our augmented $N$-reel problem. The index policy is derived from the single-reel naive MRP and incorporates long-run costs. We will use the bias function of the single-reel model, $h_1(\cdot)$, which captures the relative expected cost starting from a given reel weight. Given the current reel weights $(w_1, \ldots, w_N)$ and an incoming component of weight $x$, the index policy selects the reel that minimizes the immediate cost plus the marginal increase in long-term expected waste. Benefiting from $h_1(\cdot)$ defined on the single-reel state space $W$, the reel selected by the index policy is computed by
{
  \setlength{\abovedisplayskip}{4pt}
  \setlength{\belowdisplayskip}{4pt}
  \begin{equation}
    n^*(w_1,\ldots,w_N, x)
    = \argminA_{n \in \mathcal{N}} \Bigl\{
        \Bigl(h_1\bigl(c(w_n,x) + e(w_n,x)\bigr)\Bigr) - h_1(w_n)
      \Bigr\}
    \label{eq:n(s)}
  \end{equation}
}

This policy represents a one-step policy improvement over the random policy baseline and can be viewed as a greedy strategy with respect to the value function $h_1$. It balances immediate filament waste and future expected costs, using a decomposition that is theoretically grounded, tractable and computationally efficient. The policy summarized in Algorithm 1 is directly aligned with the MDP modeling introduced in \S\ref{sec:MDPmodel}.

\textbf{Algorithm 1} The Index Policy for the MDP formulation of the N-bounded online cutting stock problem consists of the following steps. 

Recursively calculate indexes $h_1(w)$ for all possible $w$ values using equation \eqref{eq:bellmann_h1}, then 

\begin{itemize}
    \item Step 1: A request for a component that weights $x$ arrives.
    \vspace{-.2 \baselineskip}
    \item Step 2: Determine the policy optimal reel by $n^*(w_1,...,w_N, x)$ using equation \eqref{eq:n(s)}.
    \vspace{-.2\baselineskip}
    \item Step 3: Assign the component to reel $n^*(w_1,...,w_N, x)$. Update the reel weight by $e(w,x)$ using equation \eqref{eq:e(w,x)}. Return to step 1.
\end{itemize}

In contrast to existing work, we want to formalize the connection between the random policy and the index policy, and show how the latter can be interpreted as a structured policy improvement derived from the bias function decomposition. To obtain a link between the reward of two policies to ultimately show that the index policy performs better compared to the random policy. To facilitate the presentation of the results, we introduce the following function for the augmented N-reel system, where each state is of the form $s = (w_1, \ldots, w_N, x)$. Let $g \in \mathbb{R}$ be a scalar, and $h \in \mathbb{R}^{|\mathcal{S}|}$ be the bias vector defined over all states. For a given policy $\pi$, we define the Bellman error operator:
{
  \setlength{\abovedisplayskip}{4pt}
  \setlength{\belowdisplayskip}{4pt}
\begin{equation} 
B_{\pi}(g, h) = c_{\pi} - g e + (P_{\pi} - I)h, 
\label{eqn:B(g,h)} 
\end{equation}
}

where $c_{\pi} \in \mathbb{R}^{|\mathcal{S}|}$ is the expected immediate cost vector under policy $\pi$, $P_{\pi} \in \mathbb{R}^{|\mathcal{S}| \times |\mathcal{S}|}$ is the transition matrix induced by $\pi$, $I$ is the identity matrix and $e$ is the vector with all entries equal to 1 in appropriate dimension. For the random policy $\pi(R)$, the pair $(\hat{g}_N^{\pi(R)}, \hat{h}_N^{\pi(R)})$ satisfies the Bellman equation for the augmented MRP, and hence:
{
  \setlength{\abovedisplayskip}{4pt}
  \setlength{\belowdisplayskip}{4pt}
\begin{equation*} 
B_{\pi(R)}(\hat{g}_N^{\pi(R)}, \hat{h}_N^{\pi(R)}) = 0. 
\end{equation*}
}
 We will use some results for the limiting matrix from Appendix A.4 \citet{Puterman2014} which will be useful in linking the random policy and the index policy. Let $ P_{\pi}^*$ denote the limiting matrix associated with $P_{\pi}$, defined by 
{
  \setlength{\abovedisplayskip}{4pt}
  \setlength{\belowdisplayskip}{4pt}
 \begin{equation*} 
 P_{\pi}^* = \lim_{N \to \infty} \frac{1}{N} \sum_{t=1}^N P_{\pi}^{t-1}
\end{equation*}
}
 Because the limiting transition probability from a state to another does not depend on the initial state, $P_{\pi}^*$ satisfies $g^{\pi}e = P_{\pi}^*c_{\pi}$. It also satisfies $P_{\pi}^* P_{\pi} = P_{\pi}^*$ and $P_{\pi}^*e = e$. Using these properties, we establish the following lemma:
\begin{lemma}
    Let $\pi$ be a policy with gain $g^{\pi}$ and bias $ h^{\pi}$. Let $\pi'$ be another policy with gain $g^{\pi'}$ and $P_{\pi'}^*$ be the limiting matrix associated with $P_{\pi'}$. Then the difference in average cost between two policies $\pi$ and $\pi'$ satisfies
    {
  \setlength{\abovedisplayskip}{4pt}
  \setlength{\belowdisplayskip}{4pt}
\begin{equation*}
    g^{\pi'}e - g^{\pi}e = P_{\pi'}^* B_{\pi'}(g^{\pi}, h^{\pi}),
\end{equation*}
}
where $P_{\pi'}^*$ is the limiting matrix of the Markov chain under policy $\pi'$. 
\label{lemma:limiting}
\end{lemma}
\begin{proof}
    We add and subtract $g^{\pi}e$ at the right hand side of $g^{\pi'}e = P{\pi'}^*c_{\pi'}$. Using $P_{\pi'}^*(P_{\pi'}-I) = 0$ and $P_{\pi'}^*e=e$ we get 
    {
  \setlength{\abovedisplayskip}{4pt}
  \setlength{\belowdisplayskip}{4pt}
    \begin{equation*}
    g^{\pi'}e = g^{\pi}e + P_{\pi'}^* (c_{\pi'} - g^{\pi}e + (P_{\pi'}-I)h_{\pi} ).
\end{equation*}
}
    The result can be obtained using \eqref{eqn:B(g,h)}.
\end{proof}

\begin{theorem}
    Consider the augmented $N$-reel system where the states are of the form $s = (w_1, \dots w_N,x)$. Index policy $\pi(I)$ defined in Equation \ref{eq:n(s)} yields less waste compared to the random policy $\pi(R)$. Moreover, it is obtained by one step policy iteration on the random policy.
    
\end{theorem}
\begin{proof}

By Lemma \ref{lemma:limiting} for the index policy $\pi(I)$ and the random policy $\pi(R)$, we have:
{
  \setlength{\abovedisplayskip}{4pt}
  \setlength{\belowdisplayskip}{4pt}
\begin{equation*} 
g^{\pi(I)} e - g^{\pi(R)}e = P_{\pi(I)}^* B_{\pi(I)}(\hat{g}_N^{\pi(R)}, \hat{h}_N^{\pi(R)}).
 \end{equation*}
 }
    Since all elements of $P_{\pi^{\text{I}}}^*$ are nonnegative, showing $g^{\pi(I)} e - g^{\pi(R)}e \leq  0$ would conclude the proof.
   We will first show that all elements of $B_{\pi(I)}(\hat{g}_N^{\pi(R)}, \hat{h}_N^{\pi(R)})$ are nonpositive, i.e. $B_{\pi(I)}(\hat{g}_N^{\pi(R)}, \hat{h}_N^{\pi(R)}) \leq 0 $.
     Thus, we calculate
{
  \setlength{\abovedisplayskip}{4pt}
  \setlength{\belowdisplayskip}{4pt}
\begin{equation} 
B_{\pi(I)}(\hat{g}_N^{\pi(R)}, \hat{h}_N^{\pi(R)}) = c_{\pi(I)} - \hat{g}_N^{\pi(R)} e + (P_{\pi(I)} - I)\hat{h}_N^{\pi(R)},
\label{eqn:B_index} 
\end{equation}
}
where $c_{\pi^{\text{I}}}$ is the cost vector under the index policy and $P_{\pi(I)}$ is the transition matrix induced by the index policy. Let $n^*(s)$ denote the action selected at state $s$ by the index policy according to equation \eqref{eq:n(s)} and consider the open form of \eqref{eqn:B_index} for any state $s=(w_1,\dots,w_N,x)$:
{
  \setlength{\abovedisplayskip}{4pt}
  \setlength{\belowdisplayskip}{4pt}
\begin{equation*}
\begin{aligned}
    B_{\pi(I)}(\hat{g}_N^{\pi(R)}, \hat{h}_N^{\pi(R)})(s) 
    =\; & c(w_{n^*}, x) - \hat{g}_N^{\pi(R)} + \sum_{x'} p(x') \hat{h}_N^{\pi(R)}(w_1, \ldots, e(w_{n^*}, x), \ldots, w_N, x')  \\ &  - \hat{h}_N^{\pi(R)}(w_1, \ldots, w_N, x).
    \end{aligned}
\end{equation*}
}

Substituting Corollary \ref{cor:h_Nhat+=h_1_corollary} for all $\hat{h}_N$ expressions, $B_{\pi(I)}(\hat{g}_N^{\pi(R)}, \hat{h}_N^{\pi(R)})(s) $ can be written in terms of the single-reel bias function $h_1$, where $w_n' = e(w_n,x)$ for $n=n^*$, and $w_n' =w_n$ otherwise:
{
  \setlength{\abovedisplayskip}{4pt}
  \setlength{\belowdisplayskip}{1pt}
\begin{equation}
\begin{aligned}
    B_{\pi(I)}(\hat{g}_N^{\pi(R)}, \hat{h}_N^{\pi(R)})(s)  
    =\; & c(w_{n^*}, x) - \hat{g}_N ^{\pi(R)}
    + \sum_{x'} p(x') \frac{1}{N} \sum_{n=1}^N \left( c(w_n', x') + h_1(e(w_n',x')) + (N-1) h_1(w_n') \right) \\ & 
    - \frac{1}{N} \sum_{n=1}^N \left( c(w_n, x) + h_1(e(w_n,x)) + (N-1) h_1(w_n) \right).
    \label{eq:B_1}
    \end{aligned}
\end{equation}
}
Note that by the Bellman equation \eqref{eq:bellmann_h1} for any state $w_n'$ we have
{
  \setlength{\abovedisplayskip}{4pt}
  \setlength{\belowdisplayskip}{4pt}
\begin{equation*}
    h_1(w_n') + g_1 = \sum_{x'}p(x') (c(w_{n^*}, x') + h_1(e(w_n',x')) ).
\end{equation*}
}
Substituting this in \eqref{eq:B_1} and using $\sum_{x'}p(x') = 1$ we can cancel out $h_1(w_n')$ terms and get
{
  \setlength{\abovedisplayskip}{4pt}
  \setlength{\belowdisplayskip}{4pt}
\begin{equation*}
\begin{aligned}
    B_{\pi(I)}(\hat{g}_N^{\pi(R)}, \hat{h}_N^{\pi(R)})(s) 
    =\; & c(w_{n^*}, x) - \hat{g}_N^{\pi(R)} + \frac{1}{N} \sum_{n=1}^N \left( g_1 + N h_1(w_n') \right)
    \\ & - \frac{1}{N} \sum_{n=1}^N \left( c(w_n, x) + h_1(e(w_n, x)) + (N-1) h_1(w_n) \right)
    \end{aligned}
\end{equation*}
}

We can use Lemma \ref{lemma:gN_g1} and \ref{lemma:gN_gNhat} to cancel gain terms out. Uniting the terms in a single summation we get
{
  \setlength{\abovedisplayskip}{4pt}
  \setlength{\belowdisplayskip}{4pt}
\begin{equation*}
\begin{aligned}
    B_{\pi(I)}(\hat{g}_N^{\pi(R)}, \hat{h}_N^{\pi(R)})(s)  
    = c(w_{n^*}, x) 
    - \frac{1}{N} \sum_{n=1}^N \left( c(w_n, x) + h_1(e(w_n, x)) + N h_1(w_n) - h_1(w_n) - N h_1(w_n')\right)
    \end{aligned}
\end{equation*}
}

Now we exclude the $(n^{*})^{th}$ term from the summation and write it separately.
{
  \setlength{\abovedisplayskip}{4pt}
  \setlength{\belowdisplayskip}{4pt}
\begin{equation*}
\begin{aligned}
    B_{\pi(I)}(\hat{g}_N^{\pi(R)}, \hat{h}_N^{\pi(R)})(s) 
    =\; & c(w_{n^*}, x) - \frac{1}{N} \left( c(w_{n^*}, x)  + h_1(e(w_{n^*}, x)) + N h_1(w_{n^*}) - h_1(w_{n^*}) - N h_1(w_{n^*}') \right)
    \\ &   - \frac{1}{N} \sum_{n=1, n \neq n^*}^N \left( c(w_n, x) + h_1(e(w_n, x)) + N h_1(w_n) - h_1(w_n) - N h_1(w_n') \right)
    \end{aligned}
\end{equation*}
}
Using $w_n' = e(w_n,x)$ for $n=n^*$ and $w_n' =w_n$ for $n \neq n^*$, and canceling out $N h_1(w_n)$ terms inside the summation we get
{
  \setlength{\abovedisplayskip}{4pt}
  \setlength{\belowdisplayskip}{4pt}
\begin{equation*}
\begin{aligned}
    B_{\pi(I)}(\hat{g}_N^{\pi(R)}, \hat{h}_N^{\pi(R)})(s) 
    =\; & c(w_{n^*}, x) - \frac{1}{N} \left( c(w_{n^*}, x)  + h_1(e(w_{n^*}, x)) + N h_1(w_{n^*}) - h_1(w_{n^*}) - N h_1(e(w_{n^*}, x)) \right)
    \\
     &  - \frac{1}{N} \sum_{n=1, n \neq n^*}^N \left( c(w_n, x) + h_1(e(w_n, x)) - h_1(w_n) \right) 
    \end{aligned}
\end{equation*}
}
By extending the summation domain and reorganizing, we can rewrite this as
{
  \setlength{\abovedisplayskip}{4pt}
  \setlength{\belowdisplayskip}{4pt}
\begin{equation*}
\begin{aligned}
    B_{\pi(I)}(\hat{g}_N^{\pi(R)}, \hat{h}_N^{\pi(R)})(s) 
    =\; & c(w_{n^*}, x) -  h_1(w_{n^*}) - h_1(e(w_{n^*},x)) - \frac{1}{N} \sum_{n=1}^N \left(  c(w_n,x) + h_1(e(w_n,x)) - h_1(w_n) \right)
    \end{aligned}
\end{equation*}
}
As $\left(  c(w_n,x) + h_1(e(w_n,x)) - h_1(w_n) \right)$ attains its minimum value for $n^*$ over $N$ reels, subtracting its average would yield a nonpositive value. Thus, we have $g^{\pi(I)} e - g^{\pi(R)}e \leq  0$  followed by $B_{\pi(I)}(\hat{g}_N^{\pi(R)}, \hat{h}_N^{\pi(R)})  \leq 0$.
\end{proof}

\subsection{Deep Reinforcement Learning Policy}
\label{sec:DRLPolicy}

While heuristic and index-based policies provide interpretable and computationally efficient decision rules, they may fail capturing complex dependencies across reels and over time. In particular, assignments are not independent since assigning a component to a specific reel affects future state transitions and the sufficiency of that reel for subsequent print jobs. At optimality, it may even be beneficial to reserve certain reels for specific types of components which is a strategic behavior that static rules cannot express.

To capture dynamic interactions and improve upon heuristic strategies, we employ the Deep Controlled Learning (DCL) framework developed by \citep{Temizoz2025}, which is a DRL-based approximate policy iteration algorithm that has shown to outperform results in various online decision making problems. Starting with an initial policy, this algorithm trains a Neural Network at each refinement step that maps simulated states into actions. The trained Neural Network serves as a new policy.

In our setting, we implement DCL to enhance reel assignment decisions beyond what static or greedy rules like best fit or index policies can offer. The idea is to leverage simulation data to learn a policy that accounts for both immediate costs and long-term consequences of assigning components to reels.

We initialize DCL with the Index Policy described earlier, using it to generate trajectories that reflect non-trivial baseline behavior. We employ two-step policy iteration and report the best-performing DRL policy obtained among these steps. Each NN receives features that capture both global and per-reel information. Training is performed offline using simulation environments built to mirror our MDP model. The DCL update mechanism iterates between generating state–action–cost trajectories under the current policy, and updating it via the NN to minimize the average cost over an infinite horizon. 

We experiment with two types of Multi-Layer Perceptron (MLP) neural networks. The first is a shared MLP, which uses the same feature set for all actions. This approach leverages global information on the current state and the impact of assigning the component to each reel. The features used in the shared MLP are summarized in Table \ref{tab:nn_features_MLP}.

\begin{table}[h]
    \centering
    \renewcommand{\arraystretch}{1} 
    \resizebox{0.7\columnwidth}{!}{
    \begin{tabular}{|l|l|}
        \hline
        \rowcolor[gray]{0.9} \textbf{Feature Description} & \textbf{Feature} \\ \hline
        \quad Normalized component weight & $x/B$ \\  
        \quad Normalized reel weights &  $w_n/B$ $\forall n \in \{ 1, \dots, N\}$ \\  
        \quad  Normalized next step reel weights &  $e(w_n,x)/B$ $\forall n \in \{ 1, \dots, N\}$ \\  
        \quad Indicator for immediate waste for each reel & $1_{w_n<x}$ $\in \{ 0,1\}$ $\forall n \in \{ 1, \dots, N\}$ \\ \hline
    \end{tabular}
    }
    \caption{MLP Features}
    \label{tab:nn_features_MLP}
\end{table}

The second type is an Action-specific MLP (AMLP), where features are tailored to each individual action, and a separate network is trained for each reel. This design allows the model to focus on localized information relevant to a specific reel while still incorporating some global context. Table \ref{tab:nn_features_amlp} lists the joint and action-specific features used in the AMLP. 

\begin{table}[h]
    \centering
    \renewcommand{\arraystretch}{1} 
    \resizebox{0.7\columnwidth}{!}{
    \begin{tabular}{|l|l|}
        \hline
        \rowcolor[gray]{0.9} \textbf{Feature Description} & \textbf{Feature} \\ \hline
        \textbf{Joint Features} &  \\ 
        \quad Normalized component weight & $x/B$ \\  
        \quad Normalized reel weights &  $w_n/B$ $\forall n \in \{ 1, \dots, N\}$ \\ 
        \textbf{Action specific features (for each action $n'$) } & \\
        \quad Normalized reel weight &  $w_{n'}/B$ \\ 
        \quad Normalized next step reel weight & $e(w_{n'},x)/B $ \\ 
        \quad Indicator for immediate waste & $1_{w_{n'}<x} \in \{0,1\}$ \\ 
        \hline
    \end{tabular}
    }
    \caption{AMLP Features}
    \label{tab:nn_features_amlp}
\end{table}

The neural network structure and the parameter settings used for all problem instances are shown provided in Table~\ref{tab:nn_simulation_parameters}. We adopt a fixed neural network architecture and shared training hyperparameters across all DRL models, with the exception of the horizon length. To ensure the learner can capture the impact of reel replacements on the cost, we calibrate the horizon length based on the expected number of steps required to observe multiple reel changes. Specifically, we select the horizon length such that each reel is expected to be replaced at least twice within the horizon which is reflected by $3 N \frac{B}{\mathbb{E}[X]}$.

\begin{table}[h]
    \centering
    \resizebox{0.7\columnwidth}{!}{
    \begin{tabular}{|cc|cc|}
\hline
\multicolumn{2}{|c|}{\textbf{Sampling and Simulation}} & \multicolumn{2}{c|}{\textbf{NN Structure}} \\ \hline
Number of Samples:            & 500000 & Number of Layers: & 3            \\
Number of Scenarios:          & 100     & Neurons per Layer:  & \{512, 256, 128\} \\
Horizon Length:               & $3 N \frac{B}{\mathbb{E}[X]}$      & Mini Batch Size:    & 64           \\
Length of the warm-up period: & 300       &                    &              \\ \hline
\end{tabular}
}
    \caption{Hyperparameters for simulation and neural network structure}
    \label{tab:nn_simulation_parameters}
\end{table}

\section{Results}\label{sec:results}

This section evaluates the performance of different policies for the reel assignment problem under various synthetic settings and a real world case study. We compare the random policy, first fit, best fit, index policy, and a DRL policy initiated by the index policy and trained using DCL.

We consider three synthetic component weight distributions summarized in Table~\ref{tab:cases} where all instances use a standardized reel capacity of $B=5000$ grams ensuring comparability across experiments. These distributions include both scenarios with dominant weight modes and those with more varied support, designed to reflect the diversity of print job requirements encountered in industrial 3D printing. In addition to the probability structure, standard deviation $\mathrm{Std}(X)$ and fit ratio $B/ \mathbb{E}[X]$ provide further insight into the variability and utilization potential of each distribution. 

\begin{table}[h]
\centering
\begin{tabular}{c|c|c|c|c|c|}
\cline{2-6}
                            & $X$ (gr) &
  $P(X)$ &
  $\mathbb{E}[X]$ & $\mathrm{Std}(X)$ & $B/ \mathbb{E}[X]$ \\ \hline
\multicolumn{1}{|c|}{Case 1} & [ 1016, 898, 651 ]  & [ 0.33, 0.34, 0.33 ] & 855.43 & 151.38 & 5.85 \\ \hline
\multicolumn{1}{|c|}{Case 2} & [ 1500, 1000, 500 ] & [ 0.5, 0.25, 0.25 ]  & 1125.00 & 414.58 & 4.44 \\ \hline
\multicolumn{1}{|c|}{Case 3} & [ 820, 792, 192 ]   & [ 0.61, 0.21, 0.18 ] & 701.08 & 238.77 & 7.13 \\ \hline
\end{tabular}
\caption{Component distributions used in synthetic experiments.}
\label{tab:cases}
\end{table}

The synthetic instances are designed to capture a range of component weight distributions encountered in production settings. Case 1 represents a balanced weight distribution where all component sizes appear with nearly equal probability. The standard deviation is relatively low, indicating limited variability, and the fit ratio of $5.85$ suggests that, on average, reels can accommodate around five to six components. This moderate alignment between component weights and reel capacity results in fragmentation that is less predictable. Case 2 models a highly skewed distribution, where the heaviest component dominates the arrivals, with medium and lighter components arriving less frequently. This leads to the highest standard deviation among all cases, reflecting significant variability in component sizes. This highlights the need for policies that can anticipate and manage uneven consumption of reel capacity. It has a lower fit ratio which simplifies decisions due to fewer assignments per reel. Case 3 introduces moderate skew, where one component size is dominant, but lighter components still appear with notable frequency. The standard deviation reflects this mixed structure. Despite having the highest fit ratio, this value is slightly above 7, which can actually increase the likelihood of leftover filament since the extra fraction of a component often cannot be utilized. Consequently, while Case 3 permits a greater number of component assignments per reel, it simultaneously escalates the risk of policy fragmentation when the available capacity marginally falls short of the requirements for the next component.

We also evaluate the policies using a real world component distribution obtained from an industrial 3D printing dataset. Case 4, summarized in Table~\ref{tab:cases_real}, includes a broader and more irregular support, with a small number of frequently occurring component weights and a long tail of rare types. This distribution reflects the variability and complexity of practical production environments. The component weights are generally lighter, and the expected value is $504.56$ grams, leading to a fit ratio of $9.91$, the highest among all cases. This implies that, on average, nearly ten components are assigned per reel, making early allocation decisions particularly impactful due to the increased number of assignments. However, since the fit ratio is just under the integer 10, the distribution aligns well with the reel capacity and minimizes systematic waste due to fractional misalignment. Case 4 creates a challenging setting where policies must balance frequent small allocations with the occasional need to accommodate a large component, often requiring foresight and flexibility to avoid inefficient reel usage.

\begin{table}[h]
\centering
\resizebox{0.7\columnwidth}{!}{
\begin{tabular}{l|c|c|c|c|c|}
\cline{2-6}
 &
  $X$ (gr) &
  $P(X)$ &
  $\mathbb{E}[X]$  & $\mathrm{Std}(X)$  & $B/ \mathbb{E}[X]$ \\ \hline
\multicolumn{1}{|l|}{Case 4} &
  \begin{tabular}[c]{@{}c@{}}[1098, 835, 597, 540, \\ 516, 503, 382, 152, \\ 127, 81, 71, 62]\end{tabular} &
  \begin{tabular}[c]{@{}c@{}}[ [0.11,  0.02, 0.01, 0.02, \\  0.04,  0.39, 0.33, 0.01, \\ 0.02, 0.01, 0.01, 0.03] ]\end{tabular} &
  504.56 & 244.21 & 9.91 \\ \hline
\end{tabular}
}
    \caption{Component distributions used in real world experiments.}
    \label{tab:cases_real}
\end{table}

For each case, we conduct experiments with \( N \in \{2, 3, 4, 5, 8\} \), resulting in a total of 15 synthetic and 5 real problem instances.
For the index policy, we apply a relative value iteration algorithm, consistent with Section 8.5.1 of \citet{Puterman2014}, where normalization is performed only in the final iteration.

\subsection{Numerical Results}\label{sec:NumericalResults}

We begin with synthetic experiments to compare policies under controlled settings. Table~\ref{tab:numerical_costs} reports the long run average filament waste per component for the synthetic problem instances under each policy. These results illustrate how the performance gap evolves as $N$ increases, and highlight the benefits of incorporating stochastic foresight into the assignment policy. Exact solutions are provided only for instances with a state space size remains below two million states, as solving larger instances becomes computationally intractable.

\begin{table}[h]
\centering
\resizebox{0.7\columnwidth}{!}{
\begin{tabular}{|c|c|c|c|c|c|c|c|c|}
\hline
Case &
  $N$ &
  \begin{tabular}[c]{@{}c@{}}Random\\ Policy\end{tabular} &
  \begin{tabular}[c]{@{}c@{}}First\\ Fit\end{tabular} &
  \begin{tabular}[c]{@{}c@{}}Best\\ Fit\end{tabular} &
  \begin{tabular}[c]{@{}c@{}}Index\\ Policy\end{tabular} &
  MLP &
  AMLP &
  \begin{tabular}[c]{@{}c@{}}Exact\\ Solver\end{tabular} \\ \hline
\multirow{5}{*}{Case 1} & 2 & 90.344 & 66.041 & 60.450 & 38.731 & 25.327 & 25.309 & 25.053 \\ \cline{2-9} 
                        & 3 & 90.312 & 61.308 & 60.181 & 35.177 & 15.635 & 15.569 & 13.022 \\ \cline{2-9} 
                        & 4 & 90.299 & 60.635 & 60.141 & 34.893 & 15.146 & 14.960 & -      \\ \cline{2-9} 
                        & 5 & 90.285 & 60.557 & 60.069 & 34.684 & 15.331 & 15.180 & -      \\ \cline{2-9} 
                        & 8 & 90.349 & 60.139 & 59.651 & 34.471 & 14.233 & 16.426 & -      \\ \hline
\multirow{5}{*}{Case 2} & 2 & 96.999 & 45.733 & 21.736 & 16.730 & 8.006  & 8.027  & 7.694  \\ \cline{2-9} 
                        & 3 & 96.844 & 31.204 & 6.227  & 2.701  & 0.785  & 0.760  & 0.520  \\ \cline{2-9} 
                        & 4 & 97.071 & 25.772 & 1.976  & 0.404  & 0.110  & 0.124  & 0.038  \\ \cline{2-9} 
                        & 5 & 96.941 & 22.462 & 0.654  & 0.060  & 0.031  & 0.034  & 0.003  \\ \cline{2-9} 
                        & 8 & 96.886 & 16.304 & 0.021  & 0.000  & 0.004  & 0.001  & -      \\ \hline
\multirow{5}{*}{Case 3} & 2 & 65.215 & 38.585 & 18.287 & 16.277 & 11.490 & 12.199 & 10.684 \\ \cline{2-9} 
                        & 3 & 65.166 & 28.154 & 17.840 & 13.336 & 8.766  & 8.736  & -      \\ \cline{2-9} 
                        & 4 & 65.175 & 23.294 & 17.816 & 12.695 & 7.778  & 7.777  & -      \\ \cline{2-9} 
                        & 5 & 65.160 & 20.579 & 17.792 & 12.421 & 8.130  & 7.670  & -      \\ \cline{2-9} 
                        & 8 & 65.206 & 18.298 & 17.683 & 12.231 & 11.190 & 7.832  & -      \\ \hline
\end{tabular}
}
\caption{Average filament waste per component for each policy under three synthetic cases.}
\label{tab:numerical_costs}
\end{table}

As expected, the random policy consistently performs the worst across all cases, reflecting its lack of informed decision-making. First Fit shows moderate improvement over the random baseline but remains constrained by its rigid allocation rule, particularly in Cases 1 where the component weights are more balanced across categories. This leads to inefficiencies as First Fit fails to account for the risk of fragmentation, especially when mid-sized components frequently arrive. Best Fit offers additional gains by reducing immediate fragmentation, yet its myopic nature limits further improvement, particularly in Case 1 where component weights are relatively close in size. Best Fit provides moderate gains but quickly stabilizes, especially in Cases 1 and 3, where component weights are close in size and the algorithm lacks adaptability to prevent long-term inefficiencies

The Index policy achieves substantial improvements over all heuristic baselines, particularly in skewed distributions like Case 2, where anticipation of future waste is critical. Unlike the myopic heuristics, the Index policy leverages foresight to balance current fit against the risk of future fragmentation, resulting in consistently lower waste levels. 

Across all heuristic policies, the average cost decreases as the number of reels increases, due to the greater flexibility. However, the relative gap between the Index policy and heuristics persists across all $N$, highlighting the value of strategic decision-making beyond simple fit-based rules.

The DRL-based approaches under the DCL framework include two models: the MLP policy, which maps states directly to actions, and the AMLP policy, which evaluates actions explicitly by integrating action features into the input. Both policies outperform heuristic baselines (Random, First Fit, Best Fit) across all cases, but their relative strengths differ depending on the distribution characteristics.

Compared to heuristics, MLP demonstrates robust performance in skewed distributions, particularly in Case 2, where the dominance of the heaviest component simplifies the learning problem. In such settings, MLP effectively learns to prioritize heavy components, driving waste close to zero after $N=5$. However, MLP's direct state-to-action mapping limits its precision in more complex distributions, especially when dealing with moderately skewed or lighter components, as in Case 3.

In contrast, AMLP leverages an action-based architecture, where the policy explicitly evaluates the outcomes of potential actions given the current state. This design enables AMLP to make more nuanced decisions, especially in cases like Case 3, where subtle trade-offs between immediate packing efficiency and future reel availability are crucial. As a result, AMLP achieves lower waste than MLP in these cases and closely approaches the Exact Solver performance, even at small $N$.

While the DCL policies consistently outperform the Index policy, we observe that the performance of DCL methods does not always improve with increasing $N$, unlike the Index policy, which systematically benefits from more reels. This is likely due to limitations in the training process, such as suboptimal hyperparameter settings or inadequate feature representations that fail to fully capture the expanded decision space as $N$ increases. Consequently, the learned policies may not generalize perfectly across different capacities, particularly in higher $N$ where the combinatorial complexity of decisions grows.

For the specific case of $N = 2$ reels, where the exact optimal policy can be computed, we can directly assess the nonoptimality of heuristic and learned policies. Figure~\ref{fig:nonoptimality} visualizes the average waste across methods for each synthetic case confirming that the Index policy consistently outperforms Best Fit and First Fit. The figure also reports the degree of non-optimality, defined as the relative gap between a policy’s waste and that of the exact optimal solution for the same case. In all cases, the gap for DCL policies is very small, while heuristic policies can deviate substantially from optimality. This demonstrates that the DCL approach achieves near-optimal performance in all synthetic scenarios.

\begin{figure}[h]
\centering
\includegraphics[width=0.99\textwidth]{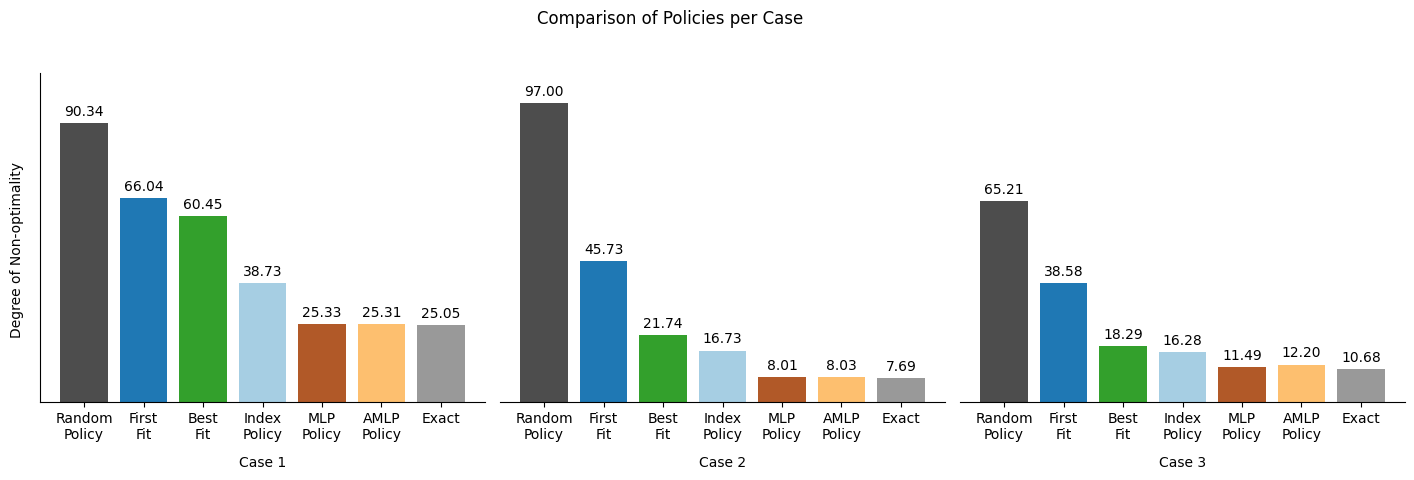}
\caption{Comparison of average filament waste under different policies for 
 $N = 2$  reels.}
\label{fig:nonoptimality}
\end{figure}

The performance gaps across methods are most pronounced in moderately skewed distributions, where achieving precise trade-offs between immediate packing and future flexibility is inherently more challenging. Nonetheless, the results for $N=2$ demonstrate that the DCL framework—particularly the AMLP variant—effectively captures the structure of the reel assignment problem in low-capacity settings. The learned policies achieve waste levels close to the optimal, without relying on handcrafted policy design or structural assumptions. This highlights the potential of data-driven approaches to generalize complex decision-making in sequential allocation problems.

\subsection{Case Study}\label{sec:CaseStudy}
We now evaluate the policies on a real-world component distribution derived from an industrial 3D printing dataset. This instance, denoted as Case 4, includes a mixture of frequently used small components and occasional large ones, reflecting the variability commonly observed in practical production settings. Due to the broad and irregular support of the distribution, even for $N=2$ the state space grows rapidly. As a result, we omit the exact solution for this case, as it becomes computationally intractable.

To examine scalability, we assess how policy performance varies with the number of available reels \( N \). Figure~\ref{fig:sensitivity_N} shows the long run average filament loss as a function of \( N \), using the same five policies as in the synthetic experiments. As expected, increasing the number of reels reduces material waste by providing greater flexibility in assignment decisions.

\begin{figure}[h]
    \centering
    \includegraphics[width=0.6\textwidth]{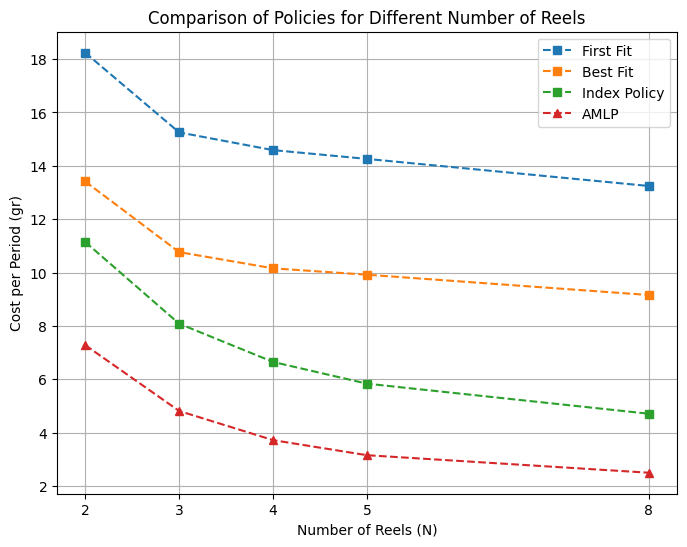}
    \caption{Average filament loss as a function of number of reels $N$ in Case 4.}
    \label{fig:sensitivity_N}
\end{figure}

Across all values of \( N \), the DRL policy consistently outperforms the alternatives. While rule-based policies such as First Fit and Best Fit also benefit from more reels, their relative performance gap remains large. The index policy scales more effectively than rule-based methods, but DRL exhibits the strongest ability to exploit larger reel sets effectively. 

As expected, more reels create less waste. However, in practice, maintaining only a small number of active reels per color is desirable to avoid clutter and complexity on the factory floor, especially given the wide variety of filament colors. Furthermore, each active reel must be fully prepared in advance which limits how many reels can be kept ready for immediate use. These setup requirements, combined with manual handling and limited workspace, make it impractical to manage a large number of active reels simultaneously. 

To determine a sufficient number of reels for implementation, we quantify the marginal benefit with respect to $N$. We apply a quantitative elbow method based on second-order differences of the DRL cost curve in Figure~\ref{fig:sensitivity_N}. Let $f(N)$ denote the average filament loss under the DRL policy with $N$ reels for Case 4. We compute the second discrete difference $\Delta^2 f(N) = f(N+1) - 2f(N) + f(N-1)$ as a proxy of curvature plotted in Figure~\ref{fig:elbow}. The second difference peaks sharply at $N=4$, and remains high at $N=5$ indicating that most of the gain in filament efficiency is captured by this range. Beyond $N=5$, the marginal improvements become smaller, suggesting diminishing returns. Hence, based on the elbow criterion, choosing $N=5$ reels would be a practically sufficient configuration that balances waste reduction and resource usage.

\begin{figure}[h]
    \centering
    \includegraphics[width=0.65\textwidth]{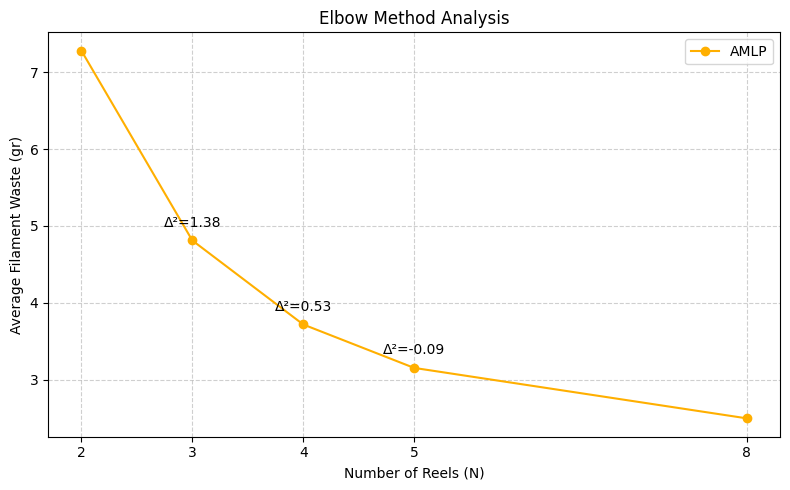}
    \caption{$\Delta^2f$ as a function of number of reels $N$}
    \label{fig:elbow}
\end{figure}

These results confirm that data driven policies offer substantial advantages over traditional heuristics in real world settings. Among them, the index policy stands out for its simplicity, transparency, theoretic background, and strong performance. By effectively leveraging stochastic information and adapting to system constraints, it consistently narrows the gap to optimal performance without requiring any training or data beyond the known job distribution. Its ability to adapt to stochastic variation and limited reel capacity makes it an attractive candidate for practical deployment. Moreover, it serves as a strong initial policy for approximate policy iteration algorithms to achieve a lower overall waste. The index policy remains a highly competitive and interpretable alternative. Finally, the elbow-based analysis indicates that a moderate number of reels is sufficient to capture most of the efficiency gains, reinforcing the practical relevance of structured online decision making in 3D printing environments.

\section{Conclusion and Discussion}\label{sec:discussion}

This work addresses the online reel assignment problem in filament-based 3D printing, modeled as an $N$-bounded online cutting-stock problem under uncertainty. We propose a sequential decision-making framework that minimizes filament waste by combining MDP-based structural insights with DRL. The proposed index policy is interpretable, theoretically grounded, and computationally efficient, while the DRL policy achieves near-optimal performance and outperforms heuristics like First Fit and Best Fit across all tested scenarios.

This work makes several contributions to both the theory and practice of sequential material allocation. First, we establish new structural results for the bounded-capacity reel assignment problem, explicitly addressing the operational reality that reels are used immediately after replacement and thus never start at full weight. Our MDP structure necessitates a novel bias function formulation within our index policy, setting it apart from existing index-based approaches. Second, we provide a closed-form expression for the index and show how one-step policy improvement effectively bridges analytical understanding with learning algorithms. Finally, we demonstrate that integrating these structural insights with deep reinforcement learning yields substantial gains in performance and scalability, enabling near-optimal decision-making in real-time production environments.

We find that the random policy performs poorly and is largely insensitive to the number of available reels, which limits its practical utility despite analytical tractability. The index policy significantly improves over baseline heuristics, especially in instances with skewed or diverse component weight distributions. Our DRL policy, initialized from the index policy, consistently achieves superior performance, confirming that learning-based methods can effectively adapt to the complexities of sequential allocation with capacity constraints. A key insight driving these results is that, unlike in classical models, reels in our setting are used immediately after replacement and never start with full weight, necessitating a tailored bias function within our index policy.

These findings establish a new bridge between analytical optimization and learning-based policies for real-time material planning in additive manufacturing. The combination of structural decomposition and one-step policy improvement provides a pathway for integrating theoretical guarantees into practical algorithms. This framework enhances operational efficiency in sequential allocation tasks where resource capacities are bounded and replenishment is not full, contributing to both the manufacturing and stochastic control literature.

Our study focuses on a specific operational setting where component orders are independent and reel replacements are immediate. While our policies are robust under these conditions, extensions to correlated component sequences, uncertainties in reel replacement, or additional operational constraints have yet to be explored. Moreover, the DRL approach, while effective, requires computational resources for training that may be a barrier for some real-world implementations.

We recommend practical deployment of the index and DRL policies in filament-based additive manufacturing settings where material waste reduction is critical. For future research, extending the framework to account for correlated demand patterns, partial reel replenishment, or dynamic operational constraints would enhance applicability. Further, exploring lightweight or online learning techniques could reduce the computational burden of DRL training, facilitating broader adoption in industry.

\bibliographystyle{elsarticle-harv} 
\bibliography{cas-refs}

\end{document}